\documentclass[twoside,a4paper,11pt]{article}
\usepackage{amsfonts, amsbsy, amsmath, amsthm, amssymb, latexsym}
\usepackage{tikz}
\usetikzlibrary{positioning}
\usepackage{mathrsfs}
\usepackage{enumerate,centernot}
\usepackage[top=28mm,right=30mm,bottom=28mm,left=30mm]{geometry}

\usepackage{bm}
\usepackage{fancyhdr}

\numberwithin{equation}{section}

\renewcommand{\bf}{\textbf}

 \renewcommand{\to}{\rightarrow}

\newcommand{\leqs}{\leqslant}

\newtheorem{theorem}{Theorem}

\newtheorem{propn}[theorem]{Proposition}

\newtheorem{thm}{Theorem}[section]

\newtheorem{prop}[thm]{Proposition}
\newtheorem{lem}[thm]{Lemma}

\newtheorem{cor}[theorem]{Corollary}

\theoremstyle{definition}

\begin{document}

\title{Alternating and symmetric groups with Eulerian generating graph}
\author{Andrea Lucchini\footnote{lucchini@math.unipd.it}\ \  and Claude Marion\footnote{marion@math.unipd.it} \\ \\ \small{\begin{tabular}{c}Dipartimento di Matematica Tullio Levi-Civita, Universit\`{a} degli Studi di Padova, Padova\end{tabular}}}

\maketitle

\noindent{\bf{Abstract.}} Given a finite  group $G$,  the generating graph  $\Gamma(G)$ of $G$ has as  vertices the (nontrivial) elements of $G$ and two vertices  are adjacent if and only if they are distinct and generate $G$ as group elements. 
In this paper we investigate  properties about the degrees of the vertices of $\Gamma(G)$ when $G$ is an alternating group or a symmetric group. In particular, we determine the vertices of $\Gamma(G)$ having even degree and show   that $\Gamma(G)$ is Eulerian if and only if $n$ and $n-1$ are not equal to a prime number congruent to 3 modulo 4.   

\section{Introduction}
Given a finite  group $G$,  the generating graph  $\Gamma(G)$ of $G$ has as  vertices the  (nontrivial) elements of $G$ and two vertices  are adjacent if and only if they are distinct and generate $G$ as group elements.\\
When $G$ is simple and $|G|>2$ many deep results on generation of $G$ in the literature can be translated to results about $\Gamma(G)$. For example, the  property   that $G$ can be generated by two elements amounts to saying that $\Gamma(G)$ has at least one edge. The fact due to  Guralnick and Kantor in \cite{GK} that every nontrivial element of $G$ belongs to a generating pair of elements of $G$ is equivalent to saying that $\Gamma(G)$ has no isolated vertices.   More recently, Breuer, Guralnick and Kantor  proved in \cite{BGK} that $G$ has spread at least 2, or in other words $\Gamma(G)$ has diameter at most 2.\\
More generally, one can try to characterise finite groups $G$ for which a given graph-theoretical property holds in  $\Gamma(G)$. 
As an illustration, recall that a graph $\Gamma$ is Hamiltonian (respectively, Eulerian) if it contains a cycle going through every vertex (respectively, edge) of $\Gamma$ exactly once. 
In \cite{BGLMN},  Breuer, Guralnick, Maroti, Nagy and the first author have investigated  the finite groups $G$ for which $\Gamma(G)$ is Hamiltonian. For example they showed that every finite simple group  of large enough order has an Hamiltonian generating graph and proposed an interesting conjecture characterising the finite groups having  an Hamiltonian generating graph.\\
It is natural to investigate the finite groups $G$ for which $\Gamma(G)$ is Eulerian. A famous result going back to Euler states that a  connected graph $\Gamma$ is Eulerian if and only if every vertex of $\Gamma$ is of even degree. \\
In this paper we study properties of finite  groups $G$ relative to the degrees of the vertices of $\Gamma(G)$.  Given an element $g\in G$, we let $|g|$ be the order of $g$ in $G$, and furthermore if $g\neq 1$ we let $\delta(g)$ be the degree of $g$ in $\Gamma(G)$. Also $G^{\textrm{ab}}$ denotes the abelianization of $G$, i.e. $G^{\textrm{ab}}=G/G'$ where $G'=[G,G]$ is the commutator subgroup of $G$.  
\\

Given a finite group $G$, our first result is a criterion for a vertex of $\Gamma(G)$ to be of even degree. 

\begin{propn}\label{p:criterion}
Let $G$ be a  finite group and $1\neq g \in G$. Let $\epsilon \in \{1,2\}$ be such that $\epsilon=1$ if and only if $G^{\textrm{ab}}$ is of odd order. If $2^\epsilon$ divides $|N_G(\langle g\rangle)|$ then $\delta(g)$ is even. 
\end{propn}

For an integer $n\geq 3$, let  ${\rm Alt}_n$ and ${\rm Sym}_n$ denote respectively the alternating and the symmetric group on $n$ letters. The rest of the paper concentrates on  $G={\rm Alt}_n$ or $G={\rm Sym}_n$. As $\Gamma({\rm Alt}_3)$  and $\Gamma({\rm Alt}_4)$  have respectively diameter 1 and 2, these graphs are both connected. In particular, the result of \cite{BGK} mentioned above yields that  $\Gamma({\rm Alt}_n)$ is connected for every $n$. \\
Binder showed in \cite{Binder} that  $\Gamma({\rm Sym}_n)$ is connected and has diameter 2 for $n>4$. In fact $\Gamma({\rm Sym}_3)$ has diameter 2 and is therefore connected but $\Gamma({\rm Sym}_4)$ is not connected as any even involution of ${\rm Sym}_4$ is an isolated vertex of $\Gamma({\rm Sym}_4)$. In particular $\Gamma({\rm Sym}_n)$ is connected for every $n\neq 4$.  \\
 Also note that ${\rm Sym}_n^{\textrm{ab}}\cong C_2$ for $n\geq 3$,   ${\rm Alt}_n^{\textrm{ab}}\cong C_3$ for $n\in \{3,4\}$ and ${\rm Alt}_n^{\textrm{ab}}$ is trivial for $n\geq 5$. (For  a positive integer $m$, $C_m$ denotes the cyclic group of order $m$.)\\

We  characterise the integers $n$  such that $|N_{{\rm Alt}_n}(\langle g\rangle)|$ is even for every $g\in {\rm Alt}_n$, as well as the integers $n$ such that $|N_{{\rm Sym}_n}(\langle g\rangle)|$ is divisible by 4 for every $g\in {\rm Sym}_n$. This gives our first examples of alternating and symmetric groups having an Eulerian generating graph.  

\begin{propn}\label{p:1}
Let $G={\rm Alt}_n$ or $G={\rm Sym}_n$ where $n\geq 3$ is an integer.  Let $\Gamma(G)$ be the generating graph of $G$ and let $1\neq g\in G$.  Let $e\in\{2,4\}$ be such that $e=2$ if $G={\rm Alt}_n$, otherwise $e=4$. The following assertions hold. 
\begin{enumerate}[(i)]
\item If $G={\rm Sym}_3$ then $|N_G(\langle g\rangle)|$ is odd only if $|g|=3$ (otherwise $|N_G(\langle g\rangle)|\equiv 2 \mod 4$.)  
 \item Suppose $G\neq {\rm Sym}_3$. Then $|N_G(\langle g\rangle)|\not \equiv 0 \mod e$  if and only if  there exists a prime number $p$ congruent to $3$ modulo $4$ and a positive integer $k$ such that $|g|=p^k$ and  $n$ can be decomposed into a sum 
\begin{equation}\label{e:sum}
n=\sum_{i=0}^k a_ip^i
\end{equation}
 where   $a_k=1$, $a_i\in \{0,1\}$ for $0\leq i< k$ and, only when $G={\rm Alt}_n$, the number of nonzero integers $a_i$ with $i$ odd is odd.   Moreover if $|N_G(\langle g \rangle)|\not \equiv 0 \mod e$ then the cycle shape of $g$ is given by $\{p^i:  a_i \neq 0\}.$
 \item  If $|g|\neq p^k$ where $p$ is a prime number congruent to $3$ modulo $4$ and $k\in \mathbb{N}$, or   $n$ cannot be decomposed into a sum as in $(\ref{e:sum})$, then $\delta(g)$ is even. 
  \item If $n$ cannot be decomposed into a sum as in $(\ref{e:sum})$ where $p$ is a prime number congruent to $3$ modulo $4$ and $k\in \mathbb{N}$, then  $\Gamma(G)$ is Eulerian. 
\end{enumerate}
\end{propn}

Given $G$,  an alternating or a  symmetric group, our main result determines the vertices of $\Gamma(G)$ of odd degree as well as whether or not $\Gamma(G)$ is Eulerian.

\begin{theorem}\label{t:6}
Let  $G={\rm Alt}_n$ or $G={\rm Sym}_n$ where $n\geq3$ is an integer.  Let $1\neq g\in G$.  
Then $\delta(g)$ is odd if and only if there exists a prime number $p$ congruent to $3$ modulo $4$ such that  $p\in\{n,n-1\}$ and $|g|=p$. In particular,  $\Gamma(G)$ is Eulerian if and only if $n$ and $n-1$ are not equal to a prime number congruent to $3$ modulo $4$. 
\end{theorem}

\begin{cor}
Let  $G={\rm Alt}_n$ or $G={\rm Sym}_n$ where $n\geq 3$ is an integer.  Suppose $\Gamma(G)$ is not Eulerian, i.e. there exists a prime number $p$ congruent to $3$ modulo $4$ such that $p\in\{n,n-1\}$. 
Let $\mathcal{P}$ be the probability that a  randomly chosen element in $\Gamma(G)$ has odd degree. 
Then $$\mathcal{P}=\frac{|{\rm Out}(G)|}{p(1-|{\rm Out}(G)|/n!)}.$$In particular, $\mathcal{P}\to 0$ as $n\to \infty$. 
\end{cor}

 We fix some notation that will be used throughout the paper. For a finite group $G$ and an element $g\in G$, we denote by  $|g|$ the order of $g$, set $\Gamma(G)$ to be the generating graph of $G$ and, if $g\neq 1$, we let  $\delta(g)$ to be the degree of $g$ in $\Gamma(G)$.  Given a positive integer $m$, we let $(\mathbb{Z}/m\mathbb{Z})^*$ be the group of units of $\mathbb{Z}/m\mathbb{Z}$ and set $\phi(m)=|(\mathbb{Z}/m\mathbb{Z})^*|$. In particular, $\phi: \mathbb{N}\rightarrow \mathbb{N}$ is the Euler's totient function.  \\
Given two positive integers $a$ and $b$, we let $(a,b)$ denote their greatest common divisor. \\
 We also use some standard group theoretical notation as set in \cite[Chapter 5, \S2]{Conway}.\\ 

 The paper is organised as follows. In \S\ref{s:2}, given a finite group $G$  and $1\neq g \in G$, we give some preliminary results on $\delta(g)$, inclusively a formula determining $\delta(g)$ via the M\"{o}bius function of $G$.
 In \S\ref{s:3} we prove Proposition \ref{p:criterion}.  In \S\ref{s:4} we prove Proposition \ref{p:1}.  In \S\ref{s:5}, we prove that if  $G={\rm Alt}_n$ or $G={\rm Sym}_n$ and $g\in G$ does not lie in a  maximal  primitive subgroup of $G$ other than ${\rm Alt}_n$ then $\delta(g)$ is even. This latter result is an important ingredient required in the proof of Theorem \ref{t:6}.  In \S\ref{s:6} we prove Theorem \ref{t:6} for $G={\rm Sym}_n$. Finally, in \S\ref{s:7} we prove Theorem \ref{t:6} for $G={\rm Alt}_n$. \\

  \noindent \textbf{Acknowledgements.} The second author thanks the MARIE CURIE and PISCOPIA research fellowship  scheme and the University of Padova  for their support.  The research leading to these results has received funding from the European Comission, Seventh Framework Programme (FP7/2007-2013) under Grant Agreement 600376. 

\section{Preliminaries}\label{s:2}

Let $G$ be a finite group and let $\Gamma(G)$ be its generating graph. Given an element $g\in G$, we let $|g|$ be the order of $g$ in $G$ and if $g\neq 1$ we let  $\delta(g)$  denote the degree of $g$ in $\Gamma(G)$. In this section, we collect some  preliminary results on $\delta(g)$. 
Some of the results rely on the M\"{o}bius function $\mu_G$ of $G$. Recall that  $\mu_G$ is the function defined on the lattice of subgroups of  $G$ by  $$\sum_{K\geq H}\mu_G(K)=\delta_{H,G},$$ where  $\delta_{G,G}=1$ and $\delta_{H,G}=0$ if $H\neq G$. We give below two general properties of $\mu_G$.

\begin{prop}\label{p:hall}
Let $G$ be a finite group. The following assertions hold:
\begin{enumerate}[(i)]
\item The M\"{o}bius function $\mu_G$ is invariant under conjugation, that is if $H_1$ and $H_2$ are conjugate in $G$ then $\mu_G(H_1)=\mu_G(H_2)$.
\item Given a subgroup $H$ of $G$, $\mu_G(H)=0$ except possibly if $H$ is an intersection of maximal subgroups of $G$ or $H=G$. 
\end{enumerate}
\end{prop}

\begin{proof}
Part (i) is clear and  part (ii) is \cite[Theorem 2.3]{Hall}. 
\end{proof}

We now record two general results on the degree of a vertex $g$ in $\Gamma(G)$.  

\begin{prop}\label{p:invol} Let $G$ be a non-cyclic finite group and $1\neq g \in G$.
Then $\delta(g)$ is even if and only if the number of involutions in $G$ adjacent to $g$ is even.
\end{prop}

\begin{proof}
Let $X$ be the set of vertices  of $\Gamma(G)$ adjacent to $g$ and let $N$ be the number of involutions adjacent to $g$.  If $X=\emptyset$ then $\delta(g)=0$ and the result is immediate. We therefore suppose that $X\neq \emptyset$. Let $x$ be any element of $X$. For any $i \in (\mathbb{Z}/|x|\mathbb{Z})^*$,  we have $\langle x\rangle= \langle x^i\rangle$ and so 
$G=\langle g, x\rangle =\langle g, x^i\rangle$. As $G$ is not cyclic,  $x^i \neq g$ and so $x^i\in X$. \\
We define a relation $\sim$ on $X$ as follows: given $x_1, x_2\in X$ we say $x_1\sim x_2$ if  and only if $x_2=x_1^i$ for some $i\in (\mathbb{Z}/|x_1|\mathbb{Z})^*$. One easily checks that $\sim$ is an equivalence relation on $X$. \\
Given $x\in X$, we let $[x]$ be the equivalence class of $X$ containing $x$ and note that $|[x]|=\phi(|x|)$.    Since $1\not \in X$, $X$ is the disjoint union of the distinct equivalence classes and $\phi(m)$ is odd if and only if $m\leq 2$, we deduce that 
$|X| \equiv N \mod 2$. In other words $\delta(g)\equiv N \mod 2$, as required. 
\end{proof}

\begin{prop}\label{p:aut}
Let $G$ be a finite group and $1\neq g\in G$. Set $$C:=C_{{\rm Aut}(G)}(g)=\{ \psi \in {\rm Aut}(G): \psi(g)=g\}.$$ Then $|C|$ divides $\delta(g)$. 
\end{prop} 

\begin{proof}
Let $X$ be the set of vertices of $\Gamma(G)$ adjacent to $g$. If $X=\emptyset$ then $\delta(g)=0$ and the result is immediate. We therefore suppose that $X\neq \emptyset$. Then $C=C_{{\rm Aut}(G)}(g)$ acts on $X$. The action is semiregular. Indeed, suppose that $\psi \in C$ is such that $\psi(x)=x$ for some $x\in X$.    Since $\psi(g)=g$ and $\langle g,x\rangle=G$, we get $\psi=1$. \\
As the action of $C$ on $X$ is semiregular, it follows from the orbit-stabiliser theorem that the size of every orbit  of $X$ under $C$ is equal to $|C|$. Hence $|C|$ divides $|X|=\delta(g)$, as required. 
\end{proof}

Given $g\in G$, we can  relate the degree $\delta(g)$ of $g$ in $\Gamma(G)$ to $\mu_G$   in the following way:

\begin{prop}\label{p:sm}
Let  $G$ be a finite group and $1\neq g \in G$. Then 
$$ \delta(g)=\sum_{H \ni g}\mu_G(H)|H|.$$
\end{prop}

\begin{proof}
This follows from \cite[(1.1)]{Lucchini}.
\end{proof}

Finally we record a result of Hawkes, Isaacs and \"{O}zaydin from  1989. In the statement below, the square-free part of a positive integer $n$ denotes the product of the distinct prime divisors of $n$.

\begin{prop}\label{p:hawkes89} \cite[Theorem 4.5]{Hawkes}.
Let $G$ be a finite group, $H$ be a subgroup of $G$ and  set $m(H)$ to be the square-free part of $|G:G'H|$. Then $|N_G(H):H|$ divides $m(H)\mu_G(H)$.  
\end{prop}

\section{Criteria for even degree}\label{s:3}

In this section we prove  Proposition \ref{p:criterion}. 

\noindent{\textit{Proof of Proposition \ref{p:criterion}.}}  Set $K=N_G(\langle g \rangle)$. 
Let $$\mathcal{S}=\{H:  H\leqs G, g \in H, |H|\equiv 1 \mod 2, \ \mu_G(H)\equiv 1 \mod 2\}.$$ 
If $\mathcal{S}=\emptyset$ then Propositions \ref{p:sm}  gives $\delta(g)\equiv 0\mod 2$.\\
We therefore suppose that $\mathcal{S}\neq \emptyset$.  
Given  a positive integer $n$, we let $n_2$  be the $2$-part of $n$.
Given $H \in \mathcal{S}$,  let $m(H)$ be the square-free part of $|G:G'H|$.  Note that for every $H\in \mathcal{S}$ we have $m(H)_2=\epsilon$.\\
By Proposition \ref{p:sm}, 
\begin{equation}\label{e:1}
\delta(g)\equiv \sum_{H \in \mathcal{S}} \mu_G(H)|H|\mod 2. 
\end{equation}
 One easily checks that $K$ acts on $\mathcal{S}$ by conjugation. Let $\mathcal{S}_i$ ($1\leq i \leq r$) be the distinct orbits of $\mathcal{S}$ under the action of  $K$, and  for $1 \leq i \leq r$, let $H_i$ be a representative of $\mathcal {S}_i$, so that $\mathcal{S}_i={\rm Orb}(H_i)$. Then   
$$ \mathcal{S}=\bigsqcup_{i=1}^r {\rm Orb}(H_i)$$ and 
\begin{equation}\label{e:2}
\sum_{H\in S} \mu_G(H)|H|=\sum_{i=1}^r  |{\rm Orb}(H_i)|  \mu_G(H_i) |H_i|.
\end{equation}
Let $1\leq i \leq r$. We have ${\rm Stab}(H_i)=K\cap N_G(H_i)$.  By the orbit-stabiliser theorem $|{\rm Orb}(H_i)|=|K:{\rm Stab}(H_i)|$ and so $|{\rm Orb}(H_i)|=|K:K\cap N_G(H_i)|$. 
Since $|H_i|$ and $\mu_G(H_i)$ are both odd, Proposition \ref{p:hawkes89} yields  that $|N_G(H_i)|_2$ divides $m(H_i)_2=\epsilon$.   
 Since  $|N_G(H_i)|_2\in\{1,\epsilon\}$ and $2^\epsilon$ divides $|K|_2$, we deduce that $|{\rm Orb}(H_i)|$ is even. \\
Equations (\ref{e:1}) and (\ref{e:2}) now give $\delta(g)\equiv 0\mod2$.
$\square$\\

\section{Normalizers of cyclic subgroups}\label{s:4}

In this section we prove Proposition \ref{p:1}. We let $G={\rm Alt}_n$ or $G={\rm Sym}_n$ where $n\geq 3$ is an integer. We first study the cases where $n\in\{3,4\}$.

\begin{prop}\label{p:34}
Let $G={\rm Alt}_n$ or $G={\rm Sym}_n$ where $n\in\{3,4\}$ and let $g\in G$. 
The following assertions hold:
\begin{enumerate}[(i)]
\item Suppose $G={\rm Alt}_3$.  The vertices of $\Gamma(G)$ correspond to the two elements of $G$ of order $3$. Moreover $\Gamma(G)$ is the complete graph $K_2$ on two vertices and is not Eulerian. Finally, $N_G(\langle g\rangle)=G$.
\item  Suppose $G={\rm Alt}_4$.  The graph $\Gamma(G)$ has eleven vertices, eight of which correspond to the elements of $G$ order $3$, the others to the involutions of $G$. 
If $|g|=2$ then $g$ is only adjacent to every element of $G$ of order $3$, whereas if $|g|=3$ then $g$ is adjacent to the nine elements in  $G\setminus \langle g\rangle$.   Moreover $\Gamma(G)$ is connected, has diameter $2$ but is not Eulerian.  Finally, $|N_G(\langle g\rangle)|\equiv 0 \mod 2$ if and only if $|g|\neq 3$. 
\item Suppose $G={\rm Sym}_3$.   The graph $\Gamma(G)$ has five vertices, two of which correspond to the elements of $G$ order $3$, the others to the involutions of $G$. 
If $|g|=2$ then $g$ is adjacent to the four elements in $G\setminus \langle g\rangle$, whereas if $|g|=3$ then $g$ is only adjacent to the three involutions of $G$.   Moreover $\Gamma(G)$ is connected, has diameter $2$ but is not Eulerian.  Finally, $|N_G(\langle g\rangle)|\equiv 0 \mod 2$ if and only if $|g|\neq 3$. 
\item Suppose $G={\rm Sym}_4$. The graph $\Gamma(G)$ has three isolated vertices corresponding to the even involutions of $G$. In particular, $\Gamma(G)$ is not Eulerian.  Finally, $|N_G(\langle g\rangle)|\equiv 0 \mod 4$ if and only if $|g|\neq 3$.
\end{enumerate}
\end{prop}

\begin{proof}
The various statements about  the order of $N_G(g)$ can be easily  checked. Also part (i) is clear. \\
We now consider part (ii). Every nontrivial element of $G$ is either an involution or an element of order 3. More precisely $G$ has three involutions and eight elements of order 3.
By Proposition \ref{p:aut} or Proposition \ref{p:sm}, if $|g|=2$ then $\delta(g)$ is even. 
  Alternatively, two involutions of $G$ generate the Klein four subgroup of $G$ and so Proposition \ref{p:invol} yields that if $|g|=2$  then $\delta(g)$ is even. In fact if $|g|=2$ then $g$ is adjacent to every element of $G$ of order 3.  Also, by Proposition \ref{p:invol} if $g\in G$ is such that $|g|=3$ then $\delta(g)$ is odd. Indeed any element of $G$ of order 3 with any of the three involutions of $G$ generate $G$.  In fact, if $|g|=3$ then $g$ is adjacent to the nine elements in $G\setminus \langle g\rangle$. \\
 We consider part (iii). Every nontrivial element of $G$ is either an involution or an element of order 3. Also $G$ has three involutions and two elements of order 3.  if $|g|=2$ then Proposition \ref{p:aut} or Proposition \ref{p:sm} yields that $\delta(g)$ is even. Alternatively, by Proposition \ref{p:invol}, if $g$ is an involution then $\delta(g)$ is even, since $g$ is adjacent to the two involutions in $G\setminus \langle g\rangle$.   In fact if $|g|=2$ then $g$ is adjacent to the four  elements in $G\setminus\langle g\rangle$.   Also by  Proposition \ref{p:invol},  if $g\in G$ is such that $|g|=3$ then $\delta(g)$ is odd. Indeed any element of $G$ of order 3 with any of the three involutions of $G$ generate $G$.  \\
 Finally part (iv) consists of an easy check. 
\end{proof}

The proof of  Proposition \ref{p:1} requires several lemmas. 

\begin{lem}\label{l:nscs}
Let $S={\rm Sym}_n$ where $n\geq 3$ is an integer. Let $g$ be an element of $S$. Set $m=|g|$ and let $D_g$ be decomposition of $g$ into disjoint  cycles.  Let $\psi: N_S(\langle g \rangle)\rightarrow (\mathbb{Z}/n\mathbb{Z})^*$  be defined as follows: for an element $s$ of $N_S(\langle g \rangle)$,  let $\psi(s)$ be  the unique element  $i_s$ in $(\mathbb{Z}/m\mathbb{Z})^*$ such that $s^{-1}gs=g^{i_s}$. The following assertions hold.
\begin{enumerate}[(i)]
\item The map $\psi$ is an epimorphism, ${\rm ker}(\psi)=C_S(g)$ and $N_S(\langle g \rangle)/C_S(g)\cong  (\mathbb{Z}/m\mathbb{Z})^*$.
\item The group $N_S(\langle g\rangle)$ is of even order.
\item If $n>3$ or $m >2$ then $|N_S(\langle g\rangle)|\equiv 0 \mod 4$ except possibly if $m=p^k$  for  some prime number $p$ congruent to $3$ modulo $4$ and some $k\in \mathbb{N}$. 
\item Suppose $n>3$ and $m=p^k$  for  some prime  number $p$ congruent to $3$ modulo $4$ and some $k\in \mathbb{N}$. If $|N_S(\langle g\rangle)|\not \equiv 0 \mod 4$ then any  two  cycles in $D_g$ have distinct lengths. 
\end{enumerate}
\end{lem}

\begin{proof}
We first consider  part (i). One easily checks that $\psi: N_S(\langle g \rangle)\rightarrow (\mathbb{Z}/n\mathbb{Z})^*$ is a surjective homomorphism and ${\rm ker}(\psi)=C_S(g)$. The result follows from the first isomorphism theorem. In particular $|N_S(\langle g\rangle)/C_S(g)|=\phi(m)$.  \\
We now consider  parts (ii) and (iii). Note that if $m=1$ then $g=1$ and $N_S(\langle g\rangle)=S$ is of even order. Moreover $S$ has order divisible by 4 for $n>3$. Also if $m=2$ then $N_S(\langle g\rangle)=C_S(g)$ contains an involution, namely $g$, and so $N_S(\langle g\rangle)$ is of even order. In fact, if $m=2$  then $N_S(\langle g \rangle)=C_S(g)$ is of order divisible by 4 for $n>3$. 
Suppose that $m>2$ is even. Then $|C_S(g)|\equiv 0 \mod 2$ and as $\phi(m)\equiv 0 \mod 2$, we get $|N_S(\langle g\rangle)|\equiv 0 \mod 4$.  \\
Suppose finally that  $m>2$ is odd. Then $\phi(m)$ is even. Also $\phi(m)\not \equiv 0 \mod 4$ if and only if $m= p^k$ for some  prime number $p$ congruent to 3 modulo 4 and some positive integer $k$. The result follows.  \\
We now consider  part (iv). By part (ii), if $|N_S(\langle g\rangle)|\not \equiv 0 \mod 4$ then $C_S(g)$ is of odd order. The result follows. 
\end{proof}

\begin{lem}\label{l:rncso}
Let $A={\rm Alt}_n$ and $S={\rm Sym}_n$ where $n> 3$ is an integer. Let $g\in A$ be an element of order $m$. Then $N_A(\langle g \rangle)$ is of even order except possibly if $m=p^k$ for  some prime number $p$ congruent to $3$ modulo $4$ and some $k\in \mathbb{N}$.
\end{lem}

\begin{proof}
 If $m=1$ or $m$ is even then $N_A(\langle g\rangle)$ contains an involution and so $N_A(\langle g\rangle)$ is of even order.
We can therefore assume that $m>1$ is odd. 
It is easy to check that $$|N_A(\langle g\rangle)/C_A(g)|=|N_S(\langle g\rangle)/C_S(g)|/\ell$$ where $\ell\in \{1,2\}$. It now follows from Lemma \ref{l:nscs} that if $\phi(m)\equiv 0 \mod 4$ then $N_A(\langle g\rangle)$
is of even order. Since $m>1$ is odd, $\phi(m)\equiv 0 \mod 4$ if and only if $m$ is not a power of a prime number congruent to 3 modulo 4.  The result follows. 
\end{proof}

\begin{lem}\label{l:ralmostfull}
Let $A={\rm Alt}_n$ and $S={\rm Sym}_n$ where $n=p^a$ for some prime number  $p$ congruent to $3$ modulo $4$ and some $a\in \mathbb{N}$. Let $g\in A$ be an $n$-cycle. 
Let $h$ be any element of $N_S(\langle g\rangle)$, say $h$ conjugates $g$ to $g^i$ for some  $i\in (\mathbb{Z}/n\mathbb{Z})^*$. Let $c_{i}$ be the order of $i$ in $(\mathbb{Z}/n\mathbb{Z})^*$.
The following assertions hold.
\begin{enumerate}[(i)]
\item $C_S(g)=\langle g\rangle$.
\item The element $h$ is of even order if and only if $c_i\equiv 0 \mod 2$.
\item Suppose $h$ is of even order. Then $h$ belongs to $A$ if and only if $a$ is even. 
\end{enumerate}
\end{lem}

\begin{proof}
Part (i) is clear. Indeed $g$ is an $n$-cycle in $S={\rm Sym}_n$, so $|C_S(g)|=n$ and $C_S(g)=\langle g \rangle$. \\
Since $|g|=n$, let $\psi: N_S(\langle g \rangle)\rightarrow (\mathbb{Z}/n\mathbb{Z})^*$  be  the map defined as in Lemma \ref{l:nscs}.  We have $$\frac{N_S(\langle g\rangle)}{\langle g\rangle}\cong {\rm Im}(\psi)\cong C_{\phi(n)}$$ where the latter isomorphism follows from the fact that $n$ is a power of a prime number. 
As $|g|$ is odd, $|h|$ is even if and only if $|\psi(h)|$ is even. This yields part (ii).\\
We now consider part (iii).  Since $n=p^a$ where $p$ is a prime number congruent to 3 modulo 4, $\phi(n)=2m$ where $m$ is an odd positive integer. In particular, as $|g|$ is odd, every Sylow 2-subgroup of $N_S(\langle g\rangle)$ has order 2.  In particular $h$ belongs to $A$ if and only if the involution $h^{|h|/2}$ belongs to $A$. Also, an involution of $N_S(\langle g\rangle)$ belongs to $A$ if and only if every involution of   $N_S(\langle g\rangle)$ belongs to $A$. Hence $h$ belongs to $A$ if and only if  the involution  of $N_S(\langle g\rangle)$ sending $g$ to $g^{-1}$ belongs to $A$. The latter happens if and only if $a$ is even. The result follows.
\end{proof}

We can now prove  Proposition \ref{p:1}.\\ 
 
\noindent{\it Proof of Proposition \ref{p:1}.}
The case where $n\in\{3,4\}$ follows directly from Proposition \ref{p:34}. We therefore assume that $n>4$. 
Note that part (iv) is an immediate consequence of part (iii). We consider parts (ii) and (iii).  Let $D_g$ be the decomposition of $g$ into disjoint cycles. We first suppose that $G={\rm Alt}_n$.\\
Suppose that $N_G(\langle g \rangle)$ is of odd order.  By Lemma \ref{l:rncso}, $|g|=p^k$ for some prime number $p$ congruent to 3 modulo 4 and some $k\in \mathbb{N}$.   
Since $|g|=p^k$, it follows that, in $D_g$, a cycle has length $p^i$ for some  integer $0\leq i \leq k$. For $0\leq i \leq k$, let $a_i$ be the number of cycles in $D_g$ of length $p^i$. Since $N_G(\langle g \rangle)$ is of odd order  we must have $a_i\in \{0,1\}$ for $0 \leq i \leq k$. Moreover, as $|g|=p^k$, $a_k=1$. Finally, suppose for a contradiction, that the number $N$ of nonzero  integers $a_i$ with $i$ odd is even. Using Lemma \ref{l:ralmostfull} one easily checks that there exists an involution in $N_G(\langle g \rangle)$ conjugating $g$ to $g^{-1}$, contradicting $N_G(\langle g \rangle)$ being of odd order.  Hence $N$ is odd.\\

Suppose now that there exist  a prime number $p$ congruent to $3$ modulo $4$,  $k\in \mathbb{N}$ and integers $a_i$ ($0\leq i \leq k$) such that: 
 $|g|=p^k$, $a_i\in\{0,1\}$, $a_k=1$ and $n=\sum_{i=0}^k{a_ip^i}$ where  the number  $N$ of nonzero integers $a_i$ with $i$ odd is odd. 
In particular, in $D_g$, there are $a_i$ cycles of length $p^i$ for $0\leq i \leq k$.  Since $N$ is odd, it follows from Lemma \ref{l:ralmostfull} that $N_G(\langle g\rangle)$ is of odd order. \\

The proof for $G={\rm Sym}_n$ is similar using Lemma \ref{l:nscs}. 
$\square$

\section{A deterministic result for alternating and symmetric groups}\label{s:5}

In this section, we provide an important ingredient needed for the proof of Theorem \ref{t:6}. 

\begin{thm}\label{t:3}
Let $G={\rm Alt}_n$ or $G={\rm Sym_n}$ where  $n>4$  is an integer. Let $1\neq g\in G$. If $g$ does not belong to a maximal primitive subgroup of $G$ not equal to ${\rm Alt}_n$ then $\delta(g)$ is even. In particular, if $n$ is  such that ${\rm Sym}_n$ and ${\rm Alt}_n$ are the only primitive groups of degree $n$, then every vertex in  $\Gamma(G)$  has even degree and so $G$ is Eulerian. 
\end{thm}

\begin{proof}
 Suppose for a contradiction that $\delta(g)$ is odd. 
 By Proposition \ref{p:1} we deduce that $|g|=p^k$ for some prime number $p$ congruent to 3 modulo 4 and some positive integer $k$. In particular $g$ belongs to ${\rm Alt}_n$.   Moreover in the decomposition $D_g$ of $g$ into disjoint cycles, all cycles have odd length and any two cycles have distinct lengths.
 Hence, without loss of generality,  we can suppose that $$g=\prod_{i=1}^r \sigma_i$$ where $r\in \mathbb{N}$ and,  for $1\leq i\leq r$, $|\sigma_i|=p^{a_i}$ where $p$ is a prime number congruent to 3 modulo 4, the sequence $(a_i)_{i=1}^r$ of nonnegative integers is strictly decreasing and $a_1\geq 1$. Moreover we can assume  that $\sigma_1=(1,2,3,\dots, p^{a_1})$. 
The assumption on $g$ implies  that a maximal subgroup of $G$ containing $g$ is either the alternating group ${\rm Alt}_n$ (if $G={\rm Sym}_n$), or an intransitive group, or an imprimitive group. \\
By Proposition \ref{p:hall}(ii), a subgroup $H$ of $G$ containing $g$ and such that $\mu_{G}(H)\neq 0$ is either  $G$, ${\rm Alt}_n$ (if $G={\rm Sym}_n$), a maximal intransitive group,  a maximal imprimitive group, or the intersection of at least two maximal subgroups of $G$ of the kind just described.\\
Note that if $r=1$ then $g=\sigma_1$, $n=p^{a_1}$ and no intransitive subgroup of  $G$ contains $g$. 
Also if $r>1$ the intersection of all maximal intransitive subgroups of ${\rm Sym}_n$ containing $g$ is the group  $$H_1=\prod_{i=1}^r {\rm Sym}_{p^{a_i}}.$$  	Since $n>4$ the group $H_1$ is of order divisible by 4.\\
Now suppose that $g$ belongs to a maximal imprimitive subgroup $K$  of ${\rm Sym}_n$. Note that $a_1\geq 2$, as otherwise  $a_1=1$, $g$ is a $p$-cycle and $n\in \{p,p+1\}$, but there is no imprimitive subgroup of ${\rm Sym}_n$ containing a $p$-cycle.  Let $\Omega=\{1,\dots,p^{a_1}\}$ be the support of $\sigma_1$ and let $B_1,\dots, B_t$ be the blocks of $K$ having nonempty intersection with $\Omega$. The integer $t$ is the smallest positive integer  such that $g^t(B_1)=B_1$. In particular $t\leq p^{a_1}$ and $t$ divides $p^{a_1}$. Hence $t=p^{\beta}$ for some integer $0\leq \beta\leq a_1$.  Without loss of generality, $i\in B_i$ for $1\leq i \leq t$. \\
Note that 
\begin{eqnarray*}
n-p^{a_1} & = & \sum_{i=2}^r p^{a_i}\\
& \leq & \sum_{i=0}^{a_2} p^{i}\\
& = & \frac{p^{a_2+1}-1}{p-1} \\  
& < & p^{a_1}. 
\end{eqnarray*}
We claim that $t\neq 1$. Indeed, suppose $t=1$. Then a block $B$ of $K$ has size at least $p^{a_1}$. Since $n-p^{a_1}<p^{a_1}$, it follows that $K$ has a single block contradicting the imprimitivity of $K$. 
We claim also that $t\neq p^{a_1}$. Indeed, suppose $t=p^{a_1}$. Then $n-t<t$, a contradiction. \\
It follows that $1\leq \beta \leq a_1-1$. Since $\beta\leq a_1-1$, we have
$$ \Gamma_1=\{i\cdot p^{a_1-1}+1:0\leq i \leq p-1\}\subseteq B_1$$
and 
$$ \Gamma_2=\{i\cdot p^{a_1-1}+2: 0 \leq i \leq p-1\} \subseteq B_2.$$
and ${\rm Sym}(\Gamma_1)\times{\rm Sym}(\Gamma_2) \cong {\rm Sym}_p\times {\rm Sym}_p$ is a subgroup of $K$ of order divisible by 4.  \\
Let $H$ be any proper subgroup  of $G$ containing $g$ and such that $\mu_{G}(H)\neq 0$.\\
If $r=a_1=1$ then  $G={\rm Sym}_n$, $H={\rm Alt}_n$  and, as $n>3$, $H$ is of even order.  
If $r=1$ and $a_1\geq 2$ then $H$ contains $({\rm Sym}(\Gamma_1)\times{\rm Sym}(\Gamma_2))\cap {\rm Alt}_n$ which is of even order. 
If $r>1$ and $a_1=1$ then $H$ contains $H_1\cap {\rm Alt}_n$ which is of even order.  
Finally if $r>1$ and $a_1>1$ then $H$ contains $H_1\cap{\rm Alt}_n\cap ({\rm Sym}(\Gamma_1)\times{\rm Sym}(\Gamma_2))={\rm Alt}_n\cap ({\rm Sym}(\Gamma_1)\times{\rm Sym}(\Gamma_2))$ which is again of even order. 
We deduce that $H$ has even order and by Proposition \ref{p:sm} it follows that $\delta(g)$ is even, contradicting our supposition that $\delta(g)$ is odd. Arguing by contradiction, we   have showed that $\delta(g)$ is even for every
$g\in \Gamma(G)$, as required. \\
The final part of the theorem now follows at once. 
\end{proof}

\section{More on symmetric groups}\label{s:6}

In this section we prove Theorem \ref{t:6} for symmetric groups. We will need a few lemmas. 

\begin{lem}\label{l:pslinan}
Let $S={\rm Sym}_n$ where $n=p^a$  for some odd prime number $p$   and some positive integer $a$. Assume $n>5$.  Suppose $H$ is an  almost simple primitive subgroup  of $S$ with ${\rm soc}(H)={\rm PSL}_d(q)$   where $d\geq 2$ and $q=r^f$ for some prime number $r$ and some positive integer $f$.  If the action of $H$ is on the set of  points of the projective  space ${\rm PG}_{d-1}(q)$ then $d$ is prime, $(d,q-1)=1$, ${\rm PSL}_d(q)={\rm PGL}_d(q)$ and  $H\leqs {\rm Alt}_n$. 
\end{lem} 

\begin{proof}
Note that $H\leqs {\rm P\Gamma L}_d(q)$, as if $d>2$ the graph automorphism of $H$ does not act on the set of points of the projective space ${\rm PG}_{d-1}(q)$. 
Also since $n$ is odd, if $d=2$ then $q=n-1$ is even and so $(d,q-1)=1$.
In fact, since $n=(q^d-1)/(q-1)$ is a power of a prime, if $d\geq 3$ then by \cite[Proposition 1]{DL}, $d$ is prime and $(d,q-1)=1$. Therefore under the assumptions of the lemma, $d$ is prime, $(d,q-1)=1 $ and ${\rm PGL}_d(q)={\rm PSL}_d(q)$. It remains to show that $H\leqs {\rm Alt}_n$.\\ 
Suppose first that $d=2$. Then $n=q+1$. As $n$ is odd, $r=2$. Also $f\neq 2$ as $n>5$.  By \cite[Table 3C]{Basile} $H\leqs {\rm Alt}_n$.\\
Suppose now that $d>2$.   
If $f$ is odd then by \cite[Table 3C]{Basile} $H \leqs {\rm Alt}_n$. Finally, we claim that $f$ is not even. Suppose otherwise and write $f=2\ell$ for some $\ell\geq 1$.
 Since $d\geq 3$ is prime and $(q^d-1)/(q-1)$ is a power of a prime, one easily checks that $(r,f,d)\not\in\{(2,2,3),(2,2,6),(2,4,3)\}$. By Zsigmondy's theorem there is a prime divisor $u_1$ of $r^{fd}-1$ not dividing $r^s-1$ for  every $1\leq s<fd$. Clearly $u_1=p$. Again, applying Zsigmondy's theorem, there is a prime divisor $u_2$ of $r^{\ell d}-1$  not dividing $r^s-1$ for every $1\leq s<\ell d$. Clearly $u_2\neq p$. 
Now $$p^a(r^{2\ell}-1)=r^{2\ell d}-1=(r^{\ell d}-1)(r^{\ell d}+1)$$ and so $u_2$ divides $r^{2\ell}-1$, a contradiction, as  $2\ell<\ell d$ (since $d\geq3$).
\end{proof}

\begin{lem}\label{l:pglinsp+1}
Let $S={\rm Sym}_{p+1}$ where $p$ is an odd prime number and let $g$ be a $p$-element of $S$. Then there is a unique transitive subgroup of $S$ isomorphic to ${\rm PGL}_2(p)$ containing $g$. 
\end{lem}

\begin{proof}
The group ${\rm PSL}_2(p)$ acts 2-transitively and faithfully  on  the set of  points of the projective space  ${\rm PG}_1(p)$ of cardinality $p+1$. In particular ${\rm PSL}_2(p)$ is a transitive subgroup of ${\rm Sym}_{p+1}$. Moreover, ${\rm PGL}_2(p)$ is a subgroup of the normalizer of   ${\rm PSL}_2(p)$ in  ${\rm Sym}_{p+1}$.
By \cite[Proposition 3.9.2]{Basile}, there is a unique conjugacy class of transitive subgroups of $S$ isomorphic to ${\rm PSL}_2(p)$. In particular there is a single conjugacy class of transitive subgroups of $S$ isomorphic to ${\rm PGL}_2(p)$. Since ${\rm PGL}_2(p)$ contains a $p$-element and all $p$-elements are conjugate in ${\rm Sym}_{p+1}$, we deduce that there is a subgroup  $H$ of ${\rm Sym}_{p+1}$ isomorphic to ${\rm PGL}_2(p)$ containing $g$.  
From the cycle shape of $g$, we have $C_{{\rm Sym}_{p+1}}(g)=C_p$ and following  Lemma \ref{l:nscs}, $N_{{\rm Sym}_{p+1}}(\langle g\rangle)$ has order $p(p-1)$. It follows that $N_{{\rm Sym}_{p+1}}(\langle g\rangle)$ is the unique parabolic subgroup $P$ of $H$ containing $g$. We claim that $P$ is not a subgroup of any other subgroup $H_0$ of ${\rm Sym}_{p+1}$ conjugate to $H$. Suppose otherwise. Then there exists $k \in {\rm Sym}_{p+1}$ such that $H^k=H_0$, $H_0\neq H$, $P\leqs H$ and $P\leqs H_0$. Let $P_0=P^k$ and note that $P_0$ is a parabolic subgroup of $H_0$. Since all parabolic subgroups of $H_0$ are conjugate, there exists $a\in H_0$ such that $P_0=P^a$.  We obtain that $ka^{-1} \in N_{{\rm Sym}_{p+1}}(P)$. Since $P=N_{{\rm Sym}_{p+1}}(\langle g\rangle)$ and $\langle g\rangle$ is a $p$-Sylow subgroup of ${\rm Sym}_{p+1}$, we in fact have $N_{{\rm Sym}_{p+1}}(P)=P$. In particular, $ka^{-1}\in P$. Since $a\in H_0$, we obtain $k\in H_0$, and $H=H_0$, a contradiction.  It follows that $H$ is the unique transitive subgroup of ${\rm Sym}_{p+1}$ isomorphic to ${\rm PGL}_2(p)$  and containing $g$. 
\end{proof}

\begin{lem}\label{l:orderdegree}
Let $S={\rm Sym}_n$ where $n\geq 3$ is an integer.  
Let $p\leq n$ be a prime  number and $g$ be a $p$-element of $S$ having no two cycles of same length.
Then $|g|>n/2$. 
\end{lem}

\begin{proof}
Say $g$ has order $|g|=p^a$ for some $a\in \mathbb{N}$. 
Clearly 
\begin{eqnarray*}
n & \leq& \sum_{i=0}^{a} p^i\\
& = & p^{a}+\sum_{i=0}^{a-1}p^i \\
& = & p^a+\frac{p^{a}-1}{p-1}\\
& < &   2p^a\\
& = & 2|g|. 
\end{eqnarray*}
\end{proof}

\begin{lem}\label{l:primitive}
Let $S={\rm Sym}_n$ where $n\geq 3$ is an integer.   Let $p\leq n$ be a prime number and $g$ be a $p$-element of $S$ having no two cycles of same length.  Suppose that $g$ belongs to a primitive subgroup $H$ of $S$. Then either $H$ is almost simple or $H$ is of affine type.  
\end{lem}

\begin{proof}
By Lemma \ref{l:orderdegree}, $|g|>n/2$. By \cite[Theorem 1]{Praeger}, the finite primitive permutation groups split into eight families: AS, HA, SD, HS, HC, CD, TW and PA. To establish  that $H$ is almost simple or of affine type amounts to showing that $H$ is of AS or HA type.   
By \cite[Theorem 5.9]{GMPS} if a primitive group of degree $n$ of type SD contains an element $h$ of order $|h|\geq n/4$ then $n=60$ and $|h|=15$.  As $|g|>n/2$, it follows that the group $H$ is not of type SD. Also as a primitive group of type HS is contained in a primitive group of type SD, $H$ cannot be of type HS. 
If $H$ is of type HC, CD, TW then by \cite[\S5.3]{GMPS} $H$ does not contain any element of order greater than $n/4$,  contradicting $|g|>n/2$. \\
 We claim finally that if $H$ is of type  PA then $H$  is in fact almost simple. Suppose indeed that $H$ is of type PA.   Then ${\rm soc}(H)=T^\ell$ for some non abelian finite simple group $T$ and for some $\ell\in \mathbb{N}$. Moreoever, $H \leqs X {\rm wr} \  {\rm Sym}_\ell$ where ${\rm soc}(X)=T$ and $X\leqs {\rm Sym}_a$ for some $a\in \mathbb{N}$ such that $n=a^\ell$. 
 We show that $\ell=1$, establishing that  $H$ is indeed  almost simple. Since $g\in H\leqs X {\rm wr}\ {\rm Sym}_\ell$, there exist $x_1,\dots,x_\ell $ in $X$ and $\sigma$ in ${\rm Sym}_{\ell}$ such that $g=(x_1,\dots,x_\ell)\sigma$. As $g$ is a $p$-element, $\sigma$ is a $p$-element of ${\rm Sym}_\ell$ of order at most $\ell$. In particular, there exists $k\in \mathbb{N}$ with  $k\leq \ell$  such that $g^k\in X^\ell$. Now $g^k$ is a $p$-element of $({\rm Sym}_a)^\ell$ and has order at most $a$.  In particular, there exists $j\in \mathbb{N}$ with  $j\leq a$  such that $g^{kj}=1$. It follows that $|g|\leq kj\leq a\ell$. Since $|g|>n/2=a^\ell/2$, we deduce that $a\ell>a^\ell/2$. That is $2\ell \geq a^{\ell-1}$.  Since  $X$ is almost simple, $a\geq 5$ and so $\ell=1$. In particular $H$ is almost simple. The result follows. 
 \end{proof}

\begin{lem}\label{l:altactset}
Let $S={\rm Sym}_n$ where $n\geq 3$ is an integer.   Let $p\leq n$ be a prime number and $g$ be a $p$-element of $S$ having no two cycles of same length.  Suppose that $g$ belongs to an almost simple primitive subgroup $H$ of $S$ with ${\rm soc}(H)={\rm Alt}_m$  for some $m\in \mathbb{N}$. If $1<k<m-1$ then the action of $H$ is not on the set of $k$-subsets from $\{1,\dots,m\}$. 
\end{lem}

\begin{proof}
Note that $m\geq 5$. Suppose that $H$ acts  on the set of $k$-subsets from $\{1,\dots,m\}$. Then $n=\binom{m}{k}$. The assumption on $k$ gives that $\binom{m}{k}\geq 2m$, that is $n\geq 2m$. Now $g\in {\rm Alt}_m$ being a $p$-element has order at most $m$ and so $|g|\leq n/2$ contradicting  Lemma \ref{l:orderdegree}.  
\end{proof}

\begin{lem}\label{l:pslactproj} 
Let $S={\rm Sym}_n$ where $n\geq 3$ is an integer.   Let $p\leq n$ be an odd prime number congruent to $3$ modulo $4$ and let $g$ be a $p$-element of $S$ having no two cycles of same length.  Suppose that $g$ belongs to an almost simple primitive subgroup $H$ of $S$ with ${\rm soc}(H)={\rm PSL}_d(q)$  for some $2\leq d\in \mathbb{N}$ and some prime power $q=r^f$. 
If the action of $H$ is on the set of  points of the projective  space ${\rm PG}_{d-1}(q)$ then one of the following holds:
\begin{enumerate}[(i)]
\item $d=2$, $q=r=p$, $n=p+1$ and $|g|=p$.  
\item $|g|=n=(q^d-1)/(q-1)$ and $r\neq p$. 
\end{enumerate}
\end{lem}

\begin{proof}
Note that from the existence of the subgroup $H$ of $G$, we must have $n\geq 5$ and by Lemma \ref{l:pslinan} $d$ is prime. 
Say $|g|=p^a$ for some $a\in \mathbb{N}$. By assumption, $$n=(q^d-1)/(q-1)=(r^{fd}-1)/(r^f-1).$$ Also by Lemma \ref{l:orderdegree}, $|g|>n/2$.  Since $g$ has odd order and a graph automorphism of ${\rm PSL}_d(q)$ is an involution,   $g \in {\rm PGL}_d(q)\rtimes \langle \psi \rangle$ where $\psi$ is a generator of the group of field automorphisms  and has order $f$. In particular there exist $g_1 {\in \rm PGL}_d(q)$ and $\psi_1 \in \langle \psi \rangle$ such that $g=(g_1,\psi_1)$. Set $h=|\psi_1|$. Note that $h$ divides $f$. Since $|g|=p^a$, $h=p^b$ for some nonnegative integer $b\leq a$.     We first show that $g\in {\rm PGL}_d(q)$.  Suppose not. Then $h>1$ and so $b\geq 1$. 
By the proof of \cite[Theorem 2.16]{GMPS}, $g_2=g^h$  belongs to  ${\rm PGL}_d(q^{1/h})$. In particular, $g$ has order dividing $h|g_2|$. Moreover, by  \cite[Corollary 2.7]{GMPS},
$$|g_2|\leq \frac{r^{\frac{f}{h}d}-1}{r^{\frac{f}{h}}-1}.$$
Since $h\geq 3$, one can check that 
$$ 2h\frac{r^{\frac{f}{h}d}-1}{r^{\frac{f}{h}}-1}\leq \frac{r^{fd}-1}{r^{f}-1}$$ unless $d=2$, $r=2$, and $h=f=3$ or $h=f=4$.  The latter does not hold as $p$ is odd. It follows that $|g|<n/2$ except possibly if $d=2$, $r=2$ and $h=f=3$. Suppose $d=2$, $r=2$ and $h=f=3$. Then ${\rm soc}(H)={\rm PSL}_2(2^3)$, $n=9$ and $|g|$ has order dividing 9. However ${\rm PGL}_2(2^3)\rtimes \langle \psi\rangle$ has no element of order $9$ and so $|g|=3$. In particular $|g|<n/2$ in all cases, a contradiction. Hence $g\in {\rm PGL}_d(q)$ as claimed.\\
Suppose first that $r=p$.  Then $g$ is a unipotent element of ${\rm PGL}_n(q)$. Assume that $d=2$. Then $|g|=p$ and $n=q+1$. In particular, if $f>1$ then $|g|\leq n/2$, a contradiction.  So $f=1$ and $q=p$.  \\
Assume that $d=3$. Then $|g|=p$ and $n=q^2+q+1$. In particular, $|g|\leq n/2$, a contradiction. 
We can therefore assume that $d\geq 5$. By \cite[Proposition 2.6]{GMPS}, $|g|\leq p^{\lceil\frac{{\rm ln}(d)}{{\rm ln}(p)}\rceil}$. Since $p\geq 3$, $|g|< p^{{\rm ln}(d)+1}$.  But $p^{{\rm ln}(d)+1}\leq  \frac{q^d-1}{2(q-1)}$ for $p\geq 3$ and $d\geq 4$, and so $|g|\leq n/2$, a contradiction.\\
Suppose now that $r\neq p$ so that $g$ is a semisimple element of ${\rm PGL}_n(q)$. Since $g$ has no two cycles of same length, the proof of \cite[Corollary 2.7]{GMPS} yields that $|g|$  divides $n=(q^d-1)/(q-1)$ or $|g|$ divides $q^{d-1}-1$.  \\
We first show that $|g|$ does not divide $q^{d-1}-1$. Suppose otherwise.  Note that if $|g|$ is not equal to $q^{d-1}-1$ then $|g|\leq n/2$, a contradiction. Hence $|g|=q^{d-1}-1$. Since $|g|$ is 
odd,  we must have $r=2$. Moreover, as $|g|$ is a power of an odd prime, applying Zsigmondy's theorem, we deduce that $d=2$. So $n=q+1$ and $|g|=q-1$. Hence $g$ must have two fixed points, a contradiction. 
Therefore $|g|$ does not divide $q^{d-1}-1$, as claimed. \\
Suppose that $|g|$ divides $n$. Note that if $|g|\neq n$ then $|g|\leq n/2$, a contradiction. Hence $|g|=n$.  The result follows. 
\end{proof}

\begin{lem}\label{l:small}
Let $S={\rm Sym}_n$ where $n\geq 3$ is a positive integer.   Let $p\leq n$ be an odd prime number and $g$ be a $p$-element of $S$ having no two cycles of same length.  Suppose that $g$ belongs to an almost simple primitive subgroup $H$ of $S$. 
Assume that if ${\rm soc}(H)={\rm PSL}_d(q)$ then the action of $H$ is not  on the set of  points of the projective  space ${\rm PG}_{d-1}(q)$.  Moreover assume that ${\rm soc}(H)\neq{\rm Alt}_n$.  Then one of the following holds:
\begin{enumerate}[(i)]
\item  ${\rm soc}(H)=M_{11}$, $|g|=11$ and $n\in\{11,12\}$.
\item ${\rm soc}(H)=M_{12}$, $|g|=11$ and $n=12$.
\item  ${\rm soc}(H)=M_{23}$, $|g|=23$ and $n=23$.
\item  ${\rm soc}(H)=M_{24}$, $|g|=23$ and $n=24$.
\item ${\rm soc}(H)={\rm PSL}_2(7)$, $|g|=7$ and $n=7$. 
\item ${\rm soc}(H)={\rm PSL}_2(11)$, $|g|=11$ and $n=11$. 
\end{enumerate}
Moreover $H\leqs {\rm Alt}_n$.
\end{lem}

\begin{proof}
Suppose first that ${\rm soc}(H)={\rm Alt}_m$ for some $m\in \mathbb{N}$.  Since $m\neq n$ we can assume from Lemma \ref{l:altactset} and  \cite[Theorem 1.3]{GMPS} that ${\rm soc}(H)={\rm Alt}_m$ where $m \in \{5,6,7,8,9\}$. From the possible cycle shapes of $g$ and using \cite[Table 6]{GMPS} which gives the possible degrees of the corresponding action of $H$, we deduce that $m=n$, a contradiction. \\
Suppose now that ${\rm soc}(H)$ is of classical type. From the possible cycle shapes of $g$ and using \cite[Table 6]{GMPS} which gives the possible degrees of the corresponding action of $H$ we deduce that ${\rm soc}(H)={\rm PSL}_2(q)$ where $q\in \{7,8,11,19\}$. Moreover, if $q=7$ then $|g|=7$ and  $n\in \{7,8\}$, if $q=8$ then $|g|=9$ and $n=9$ , if $q=11$ then $|g|=11$ and $n\in \{11,12\}$, and if $q=19$ then $|g|=19$ and $n=20$. Note that if $n=q+1$ then the action of $H$ is on the set of  points of the projective space ${\rm PG}_{d-1}(q)$.\\ 
If ${\rm soc}(H)$ is not an alternating group nor a classical group then by \cite[Theorem 1.3]{GMPS}, ${\rm soc}(H)=M_{r}$ where $r\in \{11,12,22,23,24\}$. Considering the possible cycle shapes of $g$ and \cite[Table 6]{GMPS} which gives the possible degrees of the corresponding action of $H$ we deduce that one of the cases (i)-(iv) in the statement of the lemma holds. \\
We finally check that $H \leqs {\rm Alt}_n$. Clearly ${\rm soc}(H) \leqs {\rm Alt}_n$.  We show that if $H$ is as in (i)-(vi) then $H={\rm soc}(H)$ and so $H\leqs {\rm Alt}_n$. \\
As ${\rm Out}(H)$ is trivial for $H\in\{M_{11}, M_{23}, M_{24}\}$, if $H$ is as in (i), (iii) or (iv) then $H={\rm soc}(H)$.  \\
Suppose that ${\rm soc}(H)={\rm PSL}_2(q)$ where $q\in\{7,11\}$, so that $n=q$. As $${\rm Aut}({\rm PSL}_2(q))={\rm PGL}_2(q)={\rm PSL}_2(q).2$$ and  ${\rm PGL}_2(q)$ is not a subgroup of ${\rm Sym}_n$, we deduce that $H={\rm PSL}_2(q)$ and so $H={\rm soc}(H)$.  \\
To conclude, since ${\rm Out}(M_{12})$ is of order 2 and ${\rm Sym}_{12}$  has no subgroup of order $2|M_{12}|$, we deduce that if  ${\rm soc}(H)=M_{12}$ then $H={\rm soc}(H)$.  
 \end{proof}

\begin{lem}\label{l:affine}
Let $S={\rm Sym}_n$ where $n\geq 3$ is an integer.   Let $p\leq n$ be an odd prime  number congruent to $3$ modulo $4$ and let $g$ be a $p$-element of $S$ having no two cycles of same length.  Suppose that $g$ belongs to a primitive subgroup $H$ of $S$ of affine type. Say $H\leqs {\rm AGL}_d(q)$ for some $1\leq d\in \mathbb{N}$ and some prime number $q$.  Then $|g|=p$. Moreover one of the following assertions holds.
\begin{enumerate}[(i)]
\item  $q=p$, $d=1$, $n=p$.
\item $q=2$, $n=p+1=2^d$.
\end{enumerate}
\end{lem}

\begin{proof}
Note that $n=q^d$.
Suppose that $q=p$. Then $|g|\leq pt$ where $t$ is the order of some unipotent element of ${\rm GL}_d(q)$. By \cite[Lemma 2.2]{GMPS0} we have
\begin{eqnarray*}
t &\leq  & p^{\lceil {\rm ln}(d)/{\rm ln}(p)\rceil} \\
& \leq & p^{\frac{{\rm ln}(d)}{{\rm ln}(p)}+1} \\
& \leq & dp. 
\end{eqnarray*}
It follows that $|g|\leq dp^2$. Now by Lemma \ref{l:orderdegree},  $|g|>n/2$ and so $|g|>p^d/2$. It follows that $dp^2>p^d/2$ and so $2d>p^{d-2}$.  Since $p\equiv 3 \mod 4$, the latter inequality does not hold if $d\geq 4$, or $d\geq 3$ and $p\neq 3$. In particular $d\leq 2$, or $d=p=3$.  \\
Suppose that $d=3$ and $p=3$. Then $t\leq p$ and so $|g|\leq 3^2$. But $|g|>3^3/2$, and so $|g|>3^2$, a contradiction. In particular $d\leq 2$.
We claim that $d\neq 2$. Suppose otherwise. Then $n=p^2$, $g$ is a $p^2$-cycle and $|g|=p^2$.  But ${\rm AGL}_2(p)$ has exponent $p$, and so $|g|\neq p^2$ and  $d\neq 2$. 
Hence $d=1$, $n=p$ and so $g$ is a $p$-cycle. \\
Suppose now that $q\neq p$. Since $g$ is a $p$-element and $q\neq p$, we must have $g\in {\rm GL}_d(q)$. In particular $|g|$ divides $q^d-1$. Since, by Lemma \ref{l:orderdegree}, $|g|>n/2$, we must have $|g|=q^d-1=n-1$. 
Hence there exists a positive integer $u$ such that $p^u=q^d-1$. As $p$ is odd, $q=2$ and $p^u+1=2^d$. We show that $u=1$. Suppose otherwise. Then $u>1$ and by Zsigmondy's theorem, there is a prime number dividing $p^{2u}-1$ but not $p^u-1$. So $p^u+1$ is divisible by an odd prime, a contradiction. Therefore $u=1$, $|g|=p$ and $n=p+1=2^d$. 
\end{proof}

We can now prove Theorem \ref{t:6} for symmetric groups.

\noindent{\textit{Proof of Theorem \ref{t:6} for symmetric groups.}}
If $n\in \{3,4\}$ then Theorem \ref{t:6} follows from Proposition \ref{p:34}. We therefore assume that $n>4$. 
 By Propositions \ref{p:criterion} and \ref{p:1} 
 if $|g|$ is not a power of a prime  number congruent to 3 modulo 4 or if there are two cycles of same length in the decomposition $D_g$ of $g$ into disjoint cycles, then $\delta(g)$ is even. 
We therefore assume that $n$ is such that there exists an element $h\in {\rm Sym}_n$ such that
\begin{itemize}
 \item $|h|=p^a$ for some prime number $p$ congruent to 3 modulo 4 and some  $a\in \mathbb{N}$
 \item in $D_h$ no two cycles have the same length. 
 \end{itemize}
 Without loss of generality we assume that $g$ is an element of ${\rm Sym}_n$ having the two defining properties of $h$ above, as otherwise $\delta(g)$ is even. In particular,  $|g|=p^a$ for some prime  number $p$ congruent to 3 modulo 4 and some positive integer $a$. Also, by Lemma \ref{l:orderdegree}, $|g|>n/2$. \\
 
Suppose first that $n$ and $n-1$ are not equal to a prime congruent to 3 modulo 4.  By \cite[Theorem 1.3]{GMPS} and Lemmas \ref{l:primitive}-\ref{l:small}, a proper  primitive subgroup $H$ of ${\rm Sym}_n$, with $H\neq {\rm Alt}_n$ and $g\in H$, satisfies ${\rm soc}(H)={\rm PSL}_d(q)$ for some positive integer  $d\geq 2$ and some prime power $q$. Moreover  $|g|=n=(q^d-1)/(q-1)$.  Since $n$ is a power of a prime, by Lemma \ref{l:pslinan}, we in fact have $H \leqs {\rm Alt}_n$. In particular  the only maximal primitive subgroup  of ${\rm Sym}_n$ containing $g$ is ${\rm Alt}_n$ and Theorem \ref{t:3} yields that $\delta(g)$ is even. It follows that ${\Gamma}({\rm Sym}_n)$ is Eulerian. \\

Suppose now that $n$ is equal to a prime number congruent to 3 modulo 4. 
Then $n\geq 7$.  
 Note that if $g$ belongs to a primitive group $H$ with ${\rm soc}(G)={\rm PSL}_d(q)$, then, by  \cite[Theorem 1.3]{GMPS} and Lemma \ref{l:pslinan}, $H\leqs {\rm Alt}_n$.
By \cite[Theorem 1.3]{GMPS} and Lemmas \ref{l:primitive}-\ref{l:affine},  a proper  primitive subgroup $H$ of ${\rm Sym}_n$ containing $g$ satisfies $H\leqs {\rm Alt}_n$  or $n=|g|=p$ and $H$ is a subgroup of the unique subgroup of ${\rm Sym}_n$  containing $g$ and isomorphic to ${\rm AGL}_1(p)=p:(p-1)$.   If $g\not \in {\rm AGL}_1(p)$ then the only maximal primitive subgroup  of ${\rm Sym}_n$ containing $g$ is ${\rm Alt}_n$ and Theorem \ref{t:3} yields that $\delta(g)$ is even. We therefore suppose that $n=|g|=p$. 
 In particular $g$ is a $p$-cycle and the only maximal subgroups of ${\rm Sym}_p$ containing $g$ are $H_1={\rm Alt}_p$ and $H_2={\rm AGL}_1(p)$. These two subgroups intersect in a subgroup $H_3$ of odd order $p(p-1)/2$ and of index 2 in $H_2$. In particular the lattice of subgroups $H$ of ${\rm Sym}_p$  containing $g$ and with nonzero M\"{o}bius function $\mu_{{\rm Sym}_p}(H)$ is as follows:\\

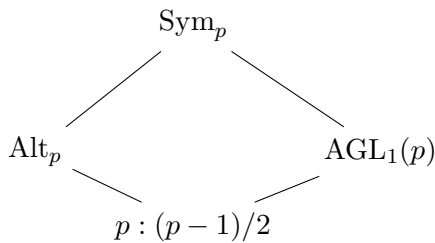
\begin{figure}[h!]
\centering
\newcommand{\mydistance}{.6cm}
\begin{tikzpicture}[node distance=2cm]
\node(Sp)                           {${\rm Sym}_p$};
\node(AG)       [below right=1cm and 1cm of Sp] {${\rm AGL}_1(p)$};
\node(A)      [below left=1cm and 1cm of Sp]  {${\rm Alt}_p$};
\node(odd)      [below=2cm of Sp]       {$p:(p-1)/2$};

\draw(Sp)       -- (AG);

    \draw (Sp) -- (A);
    \draw (A) -- (odd);
\draw(AG)--(odd);
\end{tikzpicture}
\end{figure}

Since $\mu_{{\rm Sym}_p}(H_1)=\mu_{{\rm Sym}_p}(H_2)=-1$ and $\mu_{{\rm Sym}_p}(H_3)=1$, it follows from Proposition \ref{p:sm} that 
\begin{eqnarray*}
\delta(g)& = &  \mu_{{\rm Sym}_p}({\rm Sym}_p)\cdot |{\rm Sym}_p|+ \mu_{{\rm Sym}_p}(H_1)\cdot |H_1|+\mu_{{\rm Sym}_p}(H_2)\cdot |H_2|+\mu_{{\rm Sym}_p}(H_3)\cdot |H_3|\\
& = & |{\rm Sym}_p|-|H_1|-|H_2|+|H_3|\\
& = &p!-p!/2-p(p-1)+p(p-1)/2\\
& = & p!/2-p(p-1)/2\\
& = & \frac{p(p-1)}2\cdot\left((p-2)!-1\right)  
\end{eqnarray*}
Since $p\geq 7$ and $p\equiv 3 \mod 4$, $\delta(g)$ is odd. \\

Suppose now that $n-1$ is equal to a prime congruent to 3 modulo 4. Note that $n\geq 8$. 
 Also if $g$ belongs to a primitive group $H$ with ${\rm soc}(G)={\rm PSL}_d(q)$, then, by  \cite[Theorem 1.3]{GMPS} and Lemma \ref{l:pslinan}, $H\leqs {\rm Alt}_n$. 
Moreover if $g$ belongs to a primitive subgroup $H$ of affine type contained in ${\rm AGL}_d(2)$, then, as ${\rm GL}_d(2)={\rm SL}_d(2)$ is perfect, ${\rm AGL}_d(2)$ is perfect and so $H\leqs {\rm AGL}_d(2)\leqs {\rm Alt}_n$.
By \cite[Theorem 1.3]{GMPS} and Lemmas \ref{l:primitive}-\ref{l:affine}, a proper  primitive subgroup $H$ of ${\rm Sym}_n$ containing $g$ satisfies  $H\leqs {\rm Alt}_n$  or $n-1=|g|=p$ and  $H$ is a subgroup of a transitive subgroup of  ${\rm Sym}_n$ containing $g$ and isomorphic to ${\rm PGL}_2(p)$.  By Lemma \ref{l:pglinsp+1}, the latter subgroup of ${\rm Sym}_n$ is the unique transitive subgroup of  ${\rm Sym}_n$ containing $g$ and isomorphic to ${\rm PGL}_2(p)$.  If $g\not \in {\rm PGL}_2(p)$ then the only maximal primitive subgroup  of ${\rm Sym}_n$ containing $g$ is ${\rm Alt}_n$ and Theorem \ref{t:3} yields that $\delta(g)$ is even. We therefore suppose that $n-1=|g|=p$. 
In particular $g$ is a $p$-cycle and the only maximal subgroups of ${\rm Sym}_{p+1}$ containing $g$ are ${\rm Alt}_{p+1}$, ${\rm PGL}_2(p)$ and the intransitive subgroup of ${\rm Sym}_{p+1}$ isomorphic to ${\rm Sym}_p$ and containing $g$.  The lattice of subgroups $H$ of ${\rm Sym}_{p+1}$  containing $g$ and with nonzero M\"{o}bius function $\mu_{{\rm Sym}_{p+1}}(H)$ is as follows:\\

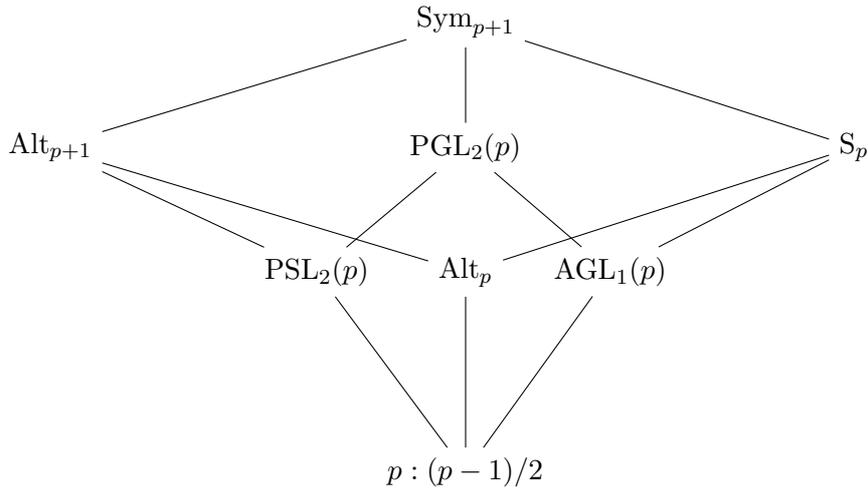
\begin{figure}[h!]\label{figure:one}
\centering
\newcommand{\mydistance}{.6cm}
\begin{tikzpicture}[node distance=2cm]
\node(Sn)                           {${\rm Sym}_{p+1}$};

\node(Sp)       [below right=1cm and 4cm of Sn] {${\rm S}_p$};
\node(A)      [below left=1cm and 4cm of Sn]  {${\rm Alt}_{p+1}$};
\node(PGL) [below=1cm of Sn]{${\rm PGL}_2(p)$};

\node(PSL)[below right=1cm and 2cm of A]{${\rm PSL}_2(p)$};
\node(AGL)[below=1cm of PGL]{${\rm Alt}_p$};
\node(Ap)[below left=1cm and 2cm of Sp]{${\rm AGL}_1(p)$};

\node(odd)[below=2cm of AGL]       {$p:(p-1)/2$};

\draw(Sn)-- (A);
\draw(Sn)--(PGL);
\draw (Sn) -- (Sp);

\draw(A)--(PSL);
\draw(PGL)--(PSL);
\draw(A)--(AGL);
\draw(Sp)--(AGL);
\draw(Sp)--(Ap);
\draw(PGL)--(Ap);

    \draw (PSL) -- (odd);
\draw(AGL)--(odd);
\draw(Ap)--(odd);
\end{tikzpicture}
\end{figure}

Using Proposition \ref{p:sm} and the fact that  $p\equiv 3 \mod 4$ and $p\geq 7$, we obtain that $$\delta(g)=\frac{p!}2\cdot p -\frac{p^2(p-1)}2$$ is odd.
$\square$

\section{More on alternating groups}\label{s:7}

\begin{prop}\label{p:smallalts7}
Let $G={\rm Alt}_n$ where $n\in\{7,11,12,23,24\}$.  Let $1\neq g \in G$. Then $\delta(g)$ is even if and only if $n\equiv 3 \mod 4$ and $|g|\neq n$, or $n\equiv 0 \mod 4$ and $|g|\neq n-1$. In particular, $\Gamma(G)$ is not Eulerian. 
\end{prop}

\begin{proof}
Suppose first that $G={\rm Alt}_7$. Let $1\neq g\in G$. By Proposition \ref{p:1}, 
$\delta(g)$ is even except possibly if $|g|=7$. Without loss of generality, we assume  that $|g|=7$. The lattice of subgroups $H$ of $G$ containing $g$ and with nonzero M\"{o}bius function $\mu_G(H)$ is as follows: 

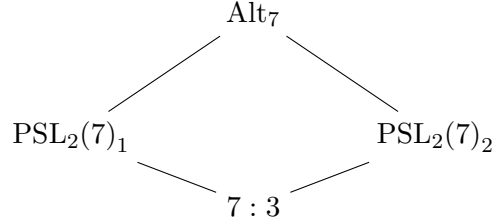
\begin{figure}[h!]
\centering
\newcommand{\mydistance}{.6cm}
\begin{tikzpicture}[node distance=2cm]
\node(Sp)                           {${\rm Alt}_7$};
\node(AG)       [below right=1cm and 1cm of Sp] {${{\rm PSL}_2(7)}_2$};
\node(A)      [below left=1cm and 1cm of Sp]  {${{\rm PSL}_2(7)}_1$};
\node(odd)      [below=2cm of Sp]       {$7:3$};

\draw(Sp)       -- (AG);

    \draw (Sp) -- (A);
    \draw (A) -- (odd);
\draw(AG)--(odd);
\end{tikzpicture}
\end{figure}
Since $\mu_G(G)=1$, $\mu_G({{\rm PSL}_2(7)}_i)=-1$ for $i \in \{1,2\}$ and $\mu_G(7:3)=1$, Proposition \ref{p:sm} yields $\delta(g)\equiv 1 \mod 2$. In particular $\Gamma(G)$ is not Eulerian. \\

Suppose that $G={\rm Alt}_{11}$. Let $1\neq g\in G$. By  Proposition \ref{p:1},  
$\delta(g)$ is even except possibly if $|g|=11$. Without loss of generality, we assume  that $|g|=11$. The lattice of subgroups $H$ of $G$ containing $g$ and with nonzero M\"{o}bius function $\mu_G(H)$ is as follows: 

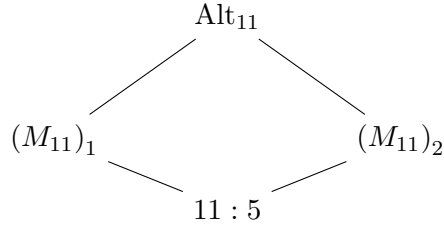
\begin{figure}[h!]
\centering
\newcommand{\mydistance}{.6cm}
\begin{tikzpicture}[node distance=2cm]
\node(Sp)                           {${\rm Alt}_{11}$};
\node(AG)       [below right=1cm and 1cm of Sp] {${(M_{11})}_2$};
\node(A)      [below left=1cm and 1cm of Sp]  {${(M_{11})}_1$};
\node(odd)      [below=2cm of Sp]       {$11:5$};

\draw(Sp)       -- (AG);

    \draw (Sp) -- (A);
    \draw (A) -- (odd);
\draw(AG)--(odd);
\end{tikzpicture}
\end{figure}

Since $\mu_G(G)=1$, $\mu_G({(M_{11})}_i)=-1$ for $i \in \{1,2\}$ and $\mu_G(11:5)=1$, Proposition \ref{p:sm} yields $\delta(g)\equiv 1 \mod 2$. In particular $\Gamma(G)$ is not Eulerian. \\

Suppose  that $G={\rm Alt}_{12}$. Let $1\neq g\in G$. By Proposition \ref{p:1},  $\delta(g)$ is even except possibly if $|g|=11$. Without loss of generality, we assume  that $|g|=11$. The lattice of subgroups $H$ of $G$ containing $g$ and with nonzero M\"{o}bius function $\mu_G(H)$ is as follows: 
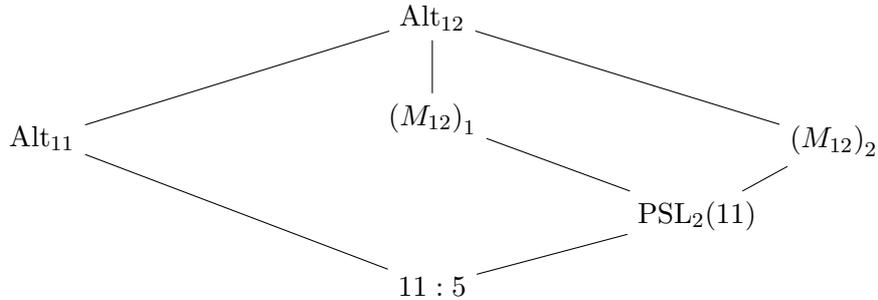
\begin{figure}[h!]
\centering
\newcommand{\mydistance}{.6cm}
\begin{tikzpicture}[node distance=2cm]
\node(Sp)                           {${\rm Alt}_{12}$};
\node(AG)       [below =1cm] {${(M_{12})}_1$};
\node(A)      [below right=1cm and 4cm of Sp]  {${(M_{12})}_2$};
\node(B) [below left=1cm and 4cm of Sp] {${\rm Alt}_{11}$};
\node(odd)      [below right =2cm  and 2cm of Sp]       {${\rm PSL}_2(11)$};
\node(C) [below =3cm of Sp]{$11:5$};

\draw(Sp)       -- (AG);

    \draw (Sp) -- (A);
    \draw (A) -- (odd);
\draw(AG)--(odd);
\draw (Sp) -- (B);
\draw (B) -- (C);
\draw (odd) -- (C);
\end{tikzpicture}
\end{figure}

Since $\mu_G(G)=1$,  $\mu_G({\rm Alt}_{11})=-1$, $\mu_G({(M_{12})}_i)=-1$ for $i \in \{1,2\}$, $\mu_G({\rm PSL}_2(11))=1$ and $\mu_G(11:5)=1$, Proposition \ref{p:sm} yields $\delta(g)\equiv 1 \mod 2$. In particular $\Gamma(G)$ is not Eulerian. \\

Suppose that $G={\rm Alt}_{23}$. Let $1\neq g\in G$. By Proposition \ref{p:1},  
$\delta(g)$ is even except possibly if $|g|=23$. Without loss of generality, we assume  that $|g|=23$. The lattice of subgroups $H$ of $G$ containing $g$ and with nonzero M\"{o}bius function $\mu_G(H)$ is as follows: 

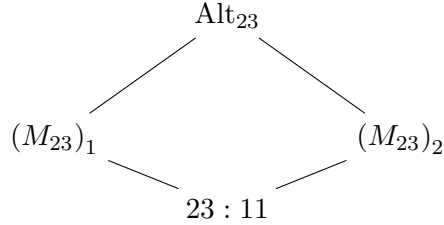
\begin{figure}[h!]
\centering
\newcommand{\mydistance}{.6cm}
\begin{tikzpicture}[node distance=2cm]
\node(Sp)                           {${\rm Alt}_{23}$};
\node(AG)       [below right=1cm and 1cm of Sp] {${(M_{23})}_2$};
\node(A)      [below left=1cm and 1cm of Sp]  {${(M_{23})}_1$};
\node(odd)      [below=2cm of Sp]       {$23:11$};

\draw(Sp)       -- (AG);

    \draw (Sp) -- (A);
    \draw (A) -- (odd);
\draw(AG)--(odd);
\end{tikzpicture}
\end{figure}

Since $\mu_G(G)=1$, $\mu_G({(M_{23})}_i)=-1$ for $i \in \{1,2\}$ and $\mu_G(23:11)=1$, Proposition \ref{p:sm} yields $\delta(g)\equiv 1 \mod 2$. In particular $\Gamma(G)$ is not Eulerian. \\

Suppose finally  that $G={\rm Alt}_{24}$. Let $1\neq g\in G$. By Proposition \ref{p:1}, 
$\delta(g)$ is even except possibly if $|g|=23$. Without loss of generality, we assume  that $|g|=23$. The lattice of subgroups $H$ of $G$ containing $g$ and with nonzero M\"{o}bius function $\mu_G(H)$ is as follows: 

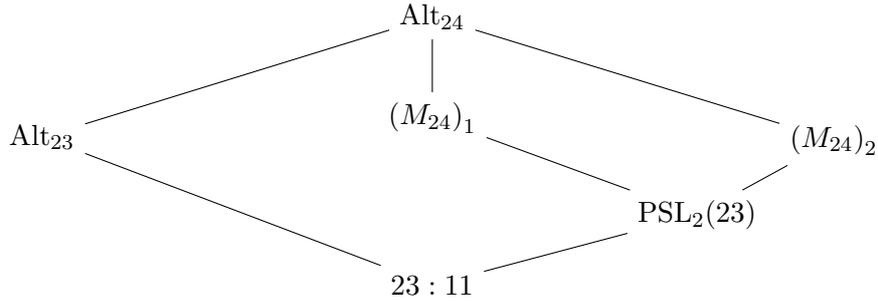
\begin{figure}[h!]
\centering
\newcommand{\mydistance}{.6cm}
\begin{tikzpicture}[node distance=2cm]
\node(Sp)                           {${\rm Alt}_{24}$};
\node(AG)       [below =1cm] {${(M_{24})}_1$};
\node(A)      [below right=1cm and 4cm of Sp]  {${(M_{24})}_2$};
\node(B) [below left=1cm and 4cm of Sp] {${\rm Alt}_{23}$};
\node(odd)      [below right =2cm  and 2cm of Sp]       {${\rm PSL}_2(23)$};
\node(C) [below =3cm of Sp]{$23:11$};

\draw(Sp)       -- (AG);

    \draw (Sp) -- (A);
    \draw (A) -- (odd);
\draw(AG)--(odd);
\draw (Sp) -- (B);
\draw (B) -- (C);
\draw (odd) -- (C);
\end{tikzpicture}
\end{figure}

Since $\mu_G(G)=1$,  $\mu_G({\rm Alt}_{23})=-1$, $\mu_G({(M_{24})}_i)=-1$ for $i \in \{1,2\}$, $\mu_G({\rm PSL}_2(23))=1$ and $\mu_G(23:11)=1$, Proposition \ref{p:sm} yields $\delta(g)\equiv 1 \mod 2$. In particular $\Gamma(G)$ is not Eulerian. \\

\end{proof}

\begin{prop}\label{p:smallcasesalt}
Let $G={\rm Alt}_n$ where $n=p^a$ for some prime number $p$ congruent to $3$ modulo $4$ and some  positive integer $a$.    
Let $1\neq g \in G$.  
 Then  $\delta(g)$ is odd if and only if $a=1$ and $|g|=n$. In particular, $\Gamma(G)$ is Eulerian if and only if $a>1$.   
\end{prop}

\begin{proof}
If $n=3$ then the result is  obvious (see Proposition \ref{p:34}). We therefore assume that $n>3$.
If $g$ does not belong to a proper primitive subgroup of $G$ then Theorem \ref{t:3} yields that $\delta(g)$ is even.  Also if $g$ is not an $n$-cycle or $a$ is even then Proposition \ref{p:1} 
yields that $\delta(g)$ is even. We therefore assume that $g$ belongs to a proper primitive subgroup of $G$, $|g|=n$ and $a$ is odd.  The cases $n\in\{7,23\}$ have been covered in Proposition \ref{p:smallalts7}, so we assume that $n>23$. 
A maximal subgroup $H$ of $G$ containing $g$ is clearly transitive, and is either imprimitive or primitive. If $H$ is  primitive then by \cite[Theorem 1.3]{GMPS}  and Lemmas \ref{l:primitive}-\ref{l:affine} either
there exist $d\in \mathbb{N}$ and a prime power $q=r^f$ (with $r$  prime and $f\in \mathbb{N}$) such that $n=|g|=(q^d-1)/(q-1)$, ${\rm soc}(H)={\rm PSL}_d(q)$ and $H={\rm P\Gamma L}_d(q)$, or 
 $n=p$ and $H$ is the subgroup $A_1$ of $G$ isomorphic to $p:(p-1)/2$ and containing $g$.  Moreover, in the former case, by Lemma \ref{l:pslinan},  $d$ must be prime, $(d,q-1)=1$ and  ${\rm PSL}_d(q)={\rm PGL}_d(q)$. We claim that in fact $d>2$.
 Suppose otherwise so that $d=2$.  Then $n=1+q$ and so, as $n$ is odd, $r=2$. Moreover since $n\equiv 3 \mod 4$,   $f$ must be odd. Now $2^f+1\equiv 0 \mod 3$.   Moreover  following Zsigmondy's theorem, $2^f+1$ is not a power of 3, unless $f\in\{1,3\}$.  However if $f\in\{1,3\}$ then $n\in\{3,9\}$, a contradiction, and so $d>2$ as claimed. \\
  Without loss of generality we can therefore assume that $n>23$, $a$ is odd, $|g|=n$ and $g$ belongs to a maximal primitive subgroup $H$ of $G$. Moreover if ${\rm soc}(H)={\rm PSL}_d(q)$, we can assume that $d>2$. \\
 
Suppose first that $n$ cannot be written as $(q^d-1)/(q-1)$ for some  prime number $d>2$ and some power $q$ of a prime. Then by assumption $n=p$ and so  $G$ has no imprimitive subgroup containing $g$.  In particular the only maximal subgroup of $G$ containing $g$ is the subgroup $A_1\cong p:(p-1)/2$ of $G$  containing $g$. Therefore, $\delta(g)=|G|-p(p-1)/2$ and  so, as $p\equiv 3 \mod 4$, $\delta(g)$ is odd. \\

Suppose now that there exist a power $q$ of a prime and a prime number $d>2$ such that $n=(q^d-1)/(q-1)$. Let $K=N_G(\langle g\rangle)$.   Then  $$|K|=n\phi(n)/2=p^a\phi(p^a)/2=p^ap^{a-1}(p-1)/2=p^ac$$
where  $c=p^{a-1}(p-1)/2$. In particular, as $p\equiv 3\mod 4$,
  $|K|$ is odd, and note 
that $K/\langle g\rangle$ is cyclic of order  $c$.  
Also if $a=1$ then $K$ is the subgroup $A_1\cong p:(p-1)/2$ of $G$  containing $g$. \\
For   a positive integer $b$ dividing $c$, let $K_b$ be the unique subgroup of $K$ containing $g$ such that $K_b/\langle g\rangle$ is cyclic of order $b$. 
Note that for every subgroup $L$ of $K$ containing $g$, there exists  a positive integer $b$ dividing  $c$  such that $L=K_b$.\\
Clearly $K$ is a transitive subgroup of $G$ and ${\rm Stab}_{K}(1)$ is a cyclic group of order $c$, say ${\rm Stab}_K(1)=\langle h \rangle$. \\
Assume $a\neq 1$. Then ${\rm Stab}_K(1)$ is not a maximal subgroup of $K$. Indeed we have the following chain of pairwise distinct subgroups of $K$. 
$$ {\rm Stab}_K(1)\subset \langle g^{p^{a-1}},h\rangle \subset K.$$  In particular if $a\neq 1$ then  $K$ is an imprimitive group. \\
Similarly, if $a\neq 1$ then $K_b$ is imprimitive for every positive integer $b$ dividing $c$. \\
Let $$\Omega=\{(q,d): q \ \textrm{is a power of a prime}, \ d \ \textrm{is an odd prime},\ n=(q^d-1)/(q-1)\}$$
and for $\omega \in \Omega$ let $\triangle_\omega$ be the family of maximal subgroups of $G$ containing $g$ and isomorphic to ${\rm P\Gamma L}_d(q)$. 
Let $\omega=(d,q)\in \Omega$. It is easy to check that $K$ acts on $\triangle_\omega$ by conjugation.   As $d>2$, $\triangle_\omega$ is the union of two orbits under the action of $K$.  As $|K|$ is odd each of these two orbits has odd length and so $\triangle_\omega$ is of even size for every $\omega \in \Omega$.\\
Also 
given $\omega\in \Omega$ and $M\in \triangle_\omega$, $K\cap M$, being the normalizer in $M$ of a Singer cycle, is a maximal subgroup of $M$, see \cite{Kantor}. Moreover, as we are in the situation where ${\rm PSL}_d(q)={\rm PGL}_d(q)$, every proper subgroup of $M$ containing $g$ is a subgroup of  $K\cap M$. \\
Let $M_1$ and $M_2$ be the representatives of the two orbits of $\triangle_\omega$ under the action of $K$. Then, for $i \in \{1,2\}$, ${\rm Stab}_K(M_i)=K\cap M_i=N_{M_i}(\langle g \rangle)$, and so ${\rm Stab}_K(M_i)$ is isomorphic to the  normalizer in ${\rm P\Gamma L}_d(q)$ of a Singer cycle. In particular  $|{\rm Stab}_K(M_1)|=|{\rm Stab}_K(M_2)|$ and so $|K\cap M_1|=|K\cap M_2|$.  
Therefore if $M\in \triangle_\omega$ then $|M\cap K|=K_{b_\omega}$ where $b_\omega$ is a positive integer dividing $c$ and depends only on $\omega$.   
It follows that given $\omega \in \Omega$ and  a positive integer $b$ dividing $p^{a-1}(p-1)/2$, the number of subgroups in $\triangle_\omega$ containing $K_b$ is even. \\

Suppose $a=1$. Then, since $n=p$ is prime, $G$ has no imprimitive subgroups.  Let $b$ be a positive integer dividing $c$. We prove by induction on $c/b$ that $\mu_G(K)=-1$  and $\mu_G(K_b)$ is even if $c/b>1$. Since $a=1$ and $n>23$,  by \cite{Liebeck}, $K$ is a maximal subgroup of $G$ and so $\mu_G(K)=-1$. Suppose $c/b>1$.
The subgroups of $G$ properly containing $K_b$ are:
\begin{enumerate}[(i)]
\item The group $G$. By definition $\mu_G(G)=1$.
\item The subgroup $K$. Note that $\mu_G(K)=-1$.
\item The subgroups in $\triangle_\omega$ for some $\omega\in \Omega$.  Note that given $\omega \in \Omega$, $|\triangle_\omega|$ is even and $\mu_G(M)=-1$ for every $M\in \triangle_\omega$. In particular 
$$\sum_{\omega\in \Omega} \sum_{M\in \triangle_\omega}\mu_G(M) \equiv 0 \mod 2$$
\item The subgroups in $S_b=\{K_e: b \ \textrm{properly divides} \ e\}$.  By induction if $H\in S_b$ then $\mu_G(H)\equiv 0 \mod 2$. In particular, $\sum_{H\in S_b} \mu_G(H)\equiv 0 \mod 2$. 
\end{enumerate}
It follows that  if $c/b>1$ then $\mu_G(K_b)\equiv 0 \mod 2$ as claimed.\\
Therefore if $H\neq K$ is a subgroup of $G$ containing $g$ then either $|H|\equiv 0 \mod 2$ or $\mu_G(H)\equiv 0 \mod 2$. It follows from Proposition \ref{p:sm} that $\delta(g)\equiv -|K| \mod 2$, and so $\delta(g)$ is odd.

Suppose finally that $a>1$.  We prove by induction on $c/b$ that $\mu_G(K_b)$ is even. Since $a>1$, $K_b$ is an imprimitive subgroup of $G$ and so $K_b$ belongs to a maximal imprimitive subgroup of $G$.  Let $I_1,\dots, I_k$ be the maximal imprimitive subgroups of $G$ containing $K_b$. Then, by the proof of Theorem \ref{t:3}, $\cap_{j=1}^k I_j$ has even order. In particular there is no positive integer $e$ dividing $c$ such that  $\cap_{j=1}^k I_j\leqs K_e$  and there is no $\omega\in \Omega$ such that $\cap_{j=1}^k I_j$ is a subgroup of an element of $\triangle_\omega$.   
The subgroups of $G$ properly containing $K_b$ are:
\begin{enumerate}[(i)]
\item The subgroups in $\triangle_\omega$ for some $\omega\in \Omega$.  Note that given $\omega \in \Omega$, $|\triangle_\omega|$ is even and $\mu_G(M)=-1$ for every $M\in \triangle_\omega$. In particular 
$$\sum_{\omega\in \Omega} \sum_{M\in \triangle_\omega}\mu_G(M) \equiv 0 \mod 2$$
\item The subgroups in $S_b=\{K_e: b \ \textrm{properly divides} \ e\}$.  By induction if $H\in S_b$ then $\mu_G(H)\equiv 0 \mod 2$. In particular, $\sum_{H\in S_b} \mu_G(H)\equiv 0 \mod 2$. 
\item The subgroups in $S=\{H\leq G: \cap_{j=1}^k I_j \leqs H \}$. We have $\sum_{H\in S}\mu_G(H)=0$. 
\end{enumerate}
It follows that for every positive integer $b$ dividing $c$, $\mu_G(K_b)\equiv 0 \mod 2$ as claimed.\\
Now a subgroup of $G$ containing $g$ is either of even order or is equal to some $K_b$ for some positive integer $b$ dividing $c$. Therefore Proposition \ref{p:sm} yields that $\delta(g)\equiv 0\mod 2$. 
\end{proof}

\begin{prop}\label{p:last}
Let $G={\rm Alt}_n$ where $n=p+1$ for some prime number $p$ congruent to $3$ modulo $4$. Let $g\in G$. Then $\delta(g)$ is odd if and only if $|g|=p$. In particular, $\Gamma(G)$ is not Eulerian.  
\end{prop}

\begin{proof} Following Proposition \ref{p:34} we can assume without loss of generality that $n\geq 8$. 
Since the cases $n\in\{12,24\}$ are treated in Proposition \ref{p:smallcasesalt}, we suppose that $n\in\{8,20\}$ or $n\geq 32$.\\ 
Let $D_g$ be the decomposition of $g$ into disjoint cycles. We can assume that  $g$ is an $r$-element for some prime number $r<n$ with $r\equiv 3 \mod 4$ and that in $D_g$ there are no two cycles of same length as otherwise, by Proposition \ref{p:1}, $|N_G(\langle g\rangle)|\equiv 0 \mod 2$ and  $\delta(g)$ is even. Note that Lemma \ref{l:orderdegree} yields that $|g|>n/2$.\\
If $g$ does not belong to a proper primitive subgroup of $G$ then Theorem \ref{t:3} yields that $\delta(g)$ is even. We therefore suppose that $g$ belongs to a proper primitive  subgroup $H$ of $G$.   By Lemma \ref{l:primitive}, $H$ is either almost simple or $H$ is affine. 
Suppose $H$ is almost simple. By \cite[Theorem 1.3]{GMPS} and  Lemma \ref{l:pslactproj}, ${\rm soc}(H)={\rm PSL}_2(p)$ and $|g|=p$. If $H$ is of affine type then it follows from \cite[Theorem 1.3]{GMPS} and Lemma \ref{l:affine} that there exists a positive integer $d$ such that $n=2^d$, $H\leqs {\rm ASL}_d(2)$ and $|g|=p$. In particular $|g|=p$. \\
Suppose that there exists no positive integer $d$ such that $n=2^d$.  The lattice of subgroups $H$ of $G$ containing $g$ and with nonzero M\"{o}bius function $\mu_G(H)$ is as follows: 

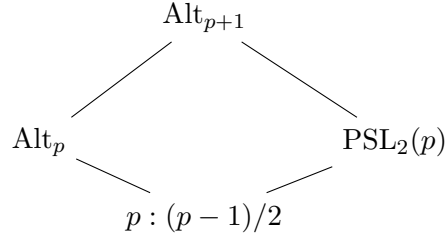
\begin{figure}[h!]
\centering
\newcommand{\mydistance}{.6cm}
\begin{tikzpicture}[node distance=2cm]
\node(Sp)                           {${\rm Alt}_{p+1}$};
\node(AG)       [below right=1cm and 1cm of Sp] {${\rm PSL}_2(p)$};
\node(A)      [below left=1cm and 1cm of Sp]  {${\rm Alt}_p$};
\node(odd)      [below=2cm of Sp]       {$p:(p-1)/2$};

\draw(Sp)       -- (AG);

    \draw (Sp) -- (A);
    \draw (A) -- (odd);
\draw(AG)--(odd);
\end{tikzpicture}
\end{figure}

Since $\mu_G(G)=1$, $\mu_G({\rm Alt}_p)=-1$, $\mu_G({\rm PSL}_2(p))=-1$ and $\mu_G(p:(p-1)/2)=1$, Proposition \ref{p:sm} yields $\delta(g)\equiv 1 \mod 2$. \\

Finally, suppose that there exists a positive integer $d$ such that $n=2^d$.  Note that $p$ is a Mersenne prime and so, since $n>4$, $d$ is an odd prime and  $2^{d-1}\equiv 1 \mod d$. Also $G$ has a unique subgroup $A_1$ isomorphic to  ${\rm Alt}_p$ and containing $g$ and a unique subgroup $P_1$ isomorphic to ${\rm PSL}_2(p)$ and containing $g$.  Let $K=N_G(\langle g\rangle)$. Then $K=p:(p-1)/2$. Now $G$ has also a  subgroup $H_1\cong {\rm ASL}_d(2)$ containing $g$.  Let   $K_1=K\cap H_1$ and $S_1=A_1\cap H_1$. Then $S_1$ is the unique subgroup of $H_1$ isomorphic to ${\rm SL}_d(2)$ and containing $g$, and $K_1=p:d=P_1\cap H_1$ is a maximal irreducible subgroup of $S_1$, namely the normalizer in $S_1$ of a Singer cycle of $S_1$. 
 In fact $G$ has $$r=\frac{(p-1)}{2d}$$ subgroups $H_i$, where $1\leq i \leq r$, such that $H_i\cong {\rm ASL}_d(2)$ and $H_i$ contains $g$.   For $1\leq i \leq r$, let $S_i=A_1\cap H_i$.  Then $S_i$ is the unique subgroup of $H_i$ isomorphic to ${\rm SL}_d(2)$ and containing $g$. Moreover, for $1\leq i, j \leq r$, $i\neq j$, we  have $H_i\cap H_j=K_1$, $P_1\cap H_i=K_1$ and $A_1\cap H_i=S_i$. 
It follows that there are only two subgroups $H$ of $G$ of odd order containing $g$ such that $\mu_G(H)$ is possibly nonzero, namely either $H=p:(p-1)/2$ or $H=p:d$. Moreover, $\mu_G(G)=\mu_G(p:(p-1)/2)=\mu_G(S_i)=1$ for 
$1\leq i \leq r$ and $\mu_G(A_1)=\mu_G(P_1)=\mu_G(H_i)=-1$ for $1\leq i \leq r$. Hence $\mu_G(p:(p-1)/2)=1$ and in fact $\mu_G(p:d)=0$.
Therefore, by Proposition \ref{p:sm}, $\delta(g)$ is odd.  
\end{proof}

We can now prove Theorem \ref{t:6} for alternating groups.

\noindent{\textit{Proof of Theorem \ref{t:6} for alternating groups}.}
The cases where $n$ is a power of a prime congruent to 3 modulo 4 or $n-1$ is a prime congruent to 3 modulo 4 have been covered (see Propositions \ref{p:smallalts7}-\ref{p:last}).
If $n$ is not a power of a prime congruent to 3 modulo 4 or $n-1$ is not a prime congruent to 3 modulo 4, then either $\delta(g)$ is even or there is no  proper primitive subgroup of ${\rm Alt}_n$ containing $g$. In the latter case, by Theorem \ref{t:3}, $\delta(g)$ is also even.


\begin{thebibliography}{99}

\bibitem{Basile} A. Basile. Second maximal subgroups of
the finite alternating and symmetric groups. PhD Thesis. The Australian National University, 2001.


\bibitem{Binder} G. Binder. The bases of the symmetric group, Izv. Vyss. Ucebn. Zaved.
Matematika \textbf{78} (1968), 19--25.

\bibitem{BGK} T. Breuer,  R. Guralnick, W. Kantor.  Probabilistic generation of finite simple groups,
II. J. Algebra  \textbf{320} (2008), 443--494.

\bibitem{BGLMN} T. Breuer, R. Guralnick, A. Lucchini, A. Mar\'{o}ti and G. Nagy. Hamiltonian cycles in the generating graphs of finite groups. Bull. London Math. Soc. \textbf{42} (2010), 621--633.


\bibitem{Conway} J. Conway, R. Curtis, S. Norton, R. Parker, R. Wilson. Atlas of  finite groups. Oxford University Press, Eynsham, 1985.
 
\bibitem{DL} E. Detomi, A. Lucchini. Profinite groups with multiplicative probabilistic zeta function. J. London Math. Soc.  \textbf{70} (2004), 165Ð181. 

\bibitem{GMPS0} S. Guest, J. Morris, C. Praeger, P. Spiga. Affine transformations of  finite vector spaces with large orders or few cycles. J. Pure Appl. Algebra \textbf{219} (2015), 308--330. 

\bibitem{GMPS} S. Guest, J. Morris, C. Praeger, P. Spiga. On the maximal orders of elements of finite almost simple groups and primitive permutation groups.  Trans. Amer. Math. Soc. \textbf{367} (2015), 7665--7694.

\bibitem{GK} R. Guralnick, W. Kantor. Probabilistic generation of finite simple groups. J. Algebra \textbf{234} (2000), 743--792.
\bibitem{Hall} P. Hall. The Eulerian functions of a group. J. Math. Oxford Ser. \textbf{7} (1936), 134--151

\bibitem{Hawkes} T. Hawkes, I. Isaacs, M. \"{O}zaydin. On the M\"{o}bius function of a finite group. Rocky Mountain Journal of Mathematics \textbf{19} (1989), 1003--1034.

\bibitem{Kantor} W. Kantor. Linear groups containing a Singer cycle. J. Algebra \textbf{62} (1980), 232--234.

\bibitem{Liebeck} M. Liebeck, C. Praeger, J. Saxl. A classification of the maximal subgroups of the finite alternating and symmetric groups. J. Algebra \textbf{111} (1987), 365--383. 

\bibitem{Lucchini} A. Lucchini. The $X$-Dirichlet polynomial of a finite group. J. Group Theory \textbf{8} (2005), 171--188.

\bibitem{Praeger} C. Praeger.  An O'Nan-Scott theorem for finite quasiprimitive permutation
groups and an application to 2-arc transitive graphs. Journal of
the London Mathematical Society \textbf{47} (1993), 227--239.
\end{thebibliography}
\end{document}